\documentclass[reqno]{amsart}
\usepackage[margin=1.2in]{geometry}
\usepackage{amsmath,amssymb,amsthm,mathtools}
\usepackage{enumitem}
\usepackage{tikz}
\usetikzlibrary{arrows.meta}
\usepackage{microtype}
\usepackage[colorlinks]{hyperref}
\hypersetup{
  colorlinks=true,
  linkcolor=blue,
  citecolor=blue,
  filecolor=blue,
  urlcolor=blue
}
\usepackage[noabbrev,capitalize]{cleveref}

\newtheorem{theorem}{Theorem}[section]
\newtheorem{proposition}[theorem]{Proposition}
\newtheorem{lemma}[theorem]{Lemma}
\newtheorem{corollary}[theorem]{Corollary}
\newtheorem{mainthm}{Theorem}

\theoremstyle{definition}
\newtheorem{definition}[theorem]{Definition}
\newtheorem{example}[theorem]{Example}
\newtheorem{remark}[theorem]{Remark}
\crefname{mainthm}{Theorem}{Theorems}

\newcommand{\Supp}{\operatorname{Supp}}
\newcommand{\id}{\operatorname{id}}
\newcommand{\Z}{\mathbb Z}
\newcommand{\Q}{\mathbb Q}
\newcommand{\lc}{\operatorname{lc}}
\newcommand{\degt}{\deg_t}

\title[Group-graded Ore extensions]{Group-graded Ore extensions and graded simplicity}
\author{Yassine Ait Mohamed}
\address{D\'epartement de math\'ematiques, Universit\'e de Sherbrooke, 2500 Bd de l'Universit\'e, Sherbrooke, QC, J1K 2R1, Canada}
\email{yassine.ait.mohamed@usherbrooke.ca}


\begin{document}
\begin{abstract}
We study Ore extensions $R[t;\sigma,\delta]$ of rings graded by an arbitrary group $\Gamma$. The grading of $R$ extends with $t$ of degree $\gamma$ if and only if $\sigma$ is compatible with conjugation by $\gamma$ and $\delta$ with left translation by $\gamma$. The extended grading is then unique. For nonabelian $\Gamma$ the degree of $t$ need not commute with the degrees of $R$, and this twist cannot always be removed. If $R$ is gr-simple and $\sigma$ is an automorphism, then $R[t;\sigma,\delta]$ is gr-simple if and only if it has no homogeneous normal element that is monic of positive $t$-degree. The quantized Weyl algebra shows that the graded forms of the classical necessary conditions do not suffice. Over a group algebra $kG$, these extensions correspond to actions of $G$ on the affine line, and gr-simplicity means that the action has no finite orbit. In characteristic zero this is the nonvanishing of a class in $H^1(G,k_\chi)$, and simplicity forces $G$ to be metabelian. We prove the analogous results for Poisson polynomial rings and relate the two settings by semiclassical limits.
\end{abstract}
\maketitle

\section{Introduction}\label{YAMQQM8W}

Ore extensions are noncommutative polynomial rings in one variable \cite{Ore1933NonCommutative}. They include skew polynomial rings, rings of differential operators and quantized Weyl algebras. Iterated Ore extensions give many of the standard examples of noncommutative noetherian rings \cite{GoodearlWarfield2004,McConnellRobson2001}. Two questions about an Ore extension $R[t;\sigma,\delta]$ recur: which structures on $R$ extend to it, and when it is simple. The second question has a long history, from Amitsur's work on differential polynomial rings over simple rings \cite{Amitsur1957} and Jordan's thesis \cite{Jordan1975Thesis} to the work of Cauchon, Lam--Leroy and \"Oinert--Richter--Silvestrov \cite{Cauchon1979,LamLeroy1992,OinertRichterSilvestrov2013}.

In this paper we study both questions for gradings by an arbitrary group. The common thread is the degree of the variable $t$. Graded Ore extensions are usually $\Z$-graded, with $\sigma$ of degree $0$ and $\delta$ of degree $1$ \cite{ShenGuo2021}. Many natural rings, however, are graded by nonabelian groups: a group algebra $kG$ is graded by $G$, and so is every crossed product by $G$. For such gradings a new phenomenon appears. The defining relation $tr=\sigma(r)t+\delta(r)$ compares $tr$ with $\sigma(r)t$, not with $rt$. So $\sigma$ may absorb conjugation by the degree of $t$, and this degree need not commute with the degrees of $R$. In the differential case $\sigma=\id$, studied in \cite{AitMohamed2026GradedDifferential}, this cannot happen. We first determine where $t$ can sit. We then describe the graded ideals. Our criterion for gr-simplicity holds for every possible degree of $t$, and the twist survives in gr-simple rings. For group algebras the twist can always be removed, and the graded ideals are read off from an action of the group on the affine line.

For example, let $a$ and $b$ be the affine maps $X\mapsto 2X$ and $X\mapsto X+1$ of $\Q$, and let $G$ be the group they generate, the Baumslag--Solitar group $BS(1,2)$. The ring generated by $\Q G$ and $t$ with
\[
 ta=2at,\quad tb=b^2t+ab
\]
is a simple Ore extension of $\Q G$, graded by $G$ with $t$ of the noncentral degree $a$. Here $\sigma$ is not a graded map, since $\sigma(b)=b^2$ has degree $aba^{-1}\neq b$. The ring is simple because $G$ acts faithfully on the affine line with no finite orbit in $\bar{\Q}$. For the free group on $a$ and $b$, the analogous ring is only gr-simple, because the free group does not embed in the affine group.

All rings are associative with identity and nonzero, ring homomorphisms preserve the identity, and ideals are two-sided unless stated otherwise. Throughout, $k$ is a field. We write $\operatorname{char}R=0$ if $n1_R\neq0$ for every integer $n\ge1$. Let $R$ be a ring and $\sigma$ a ring endomorphism of $R$. A \emph{$\sigma$-derivation} is an additive map $\delta\colon R\to R$ such that $\delta(rs)=\delta(r)s+\sigma(r)\delta(s)$ for all $r,s\in R$. The \emph{Ore extension} $A=R[t;\sigma,\delta]$ is the free left $R$-module with basis $1,t,t^2,\ldots$, with multiplication determined by
\begin{equation}\label{YAMO47UXRT}
    tr=\sigma(r)t+\delta(r)
    \quad (r\in R).
\end{equation}
We write $R[t;\sigma]$ when $\delta=0$ and $R[t;\delta]$ when $\sigma=\id$. For the general theory, see \cite{GoodearlWarfield2004,McConnellRobson2001}. Every $f\in A$ can be written uniquely as $f=\sum_ia_it^i$ with $a_i\in R$. If $f\neq0$, its \emph{$t$-degree} $\degt f$ is the largest $n$ with $a_n\neq0$, and $\lc(f)=a_n$ is its \emph{leading coefficient}. For $n\ge1$, the element $f$ is \emph{monic of $t$-degree $n$} if $f=t^n+\sum_{i<n}a_it^i$. An element $f$ of a ring $S$ is \emph{normal} if $Sf=fS$. Then $Sf$ is an ideal.

For $u\in R$, the map $\delta_u(r)=ur-\sigma(r)u$ is a $\sigma$-derivation. A $\sigma$-derivation is \emph{inner} if it equals $\delta_u$ for some $u\in R$, and \emph{outer} otherwise. If $\sigma$ is an automorphism, each $\sigma^j\delta\sigma^{-j}$ is again a $\sigma$-derivation, and for $n\ge1$ we put
\[
\Delta_n=\sum_{j=0}^{n-1}\sigma^j\delta\sigma^{-j}.
\]
This is a $\sigma$-derivation, and $\Delta_n=n\delta$ if $\sigma=\id$.

Throughout, $\Gamma$ is a group with identity $e$, not necessarily abelian. A \emph{$\Gamma$-graded ring} is a ring $R=\bigoplus_{\tau\in\Gamma}R_\tau$ with $R_\rho R_\tau\subseteq R_{\rho\tau}$ for all $\rho,\tau$ (see \cite{NastasescuVanOystaeyen2004}). The elements of $R_\tau$ are \emph{homogeneous of degree $\tau$}, and $\Supp(R)=\{\tau\in\Gamma:R_\tau\neq0\}$ is the \emph{support} of $R$. We write $\deg$ for the $\Gamma$-degree and $\degt$ for the $t$-degree. An ideal is \emph{graded} if it is generated by homogeneous elements, and a graded ring is \emph{gr-simple} if $0$ and the ring itself are its only graded ideals. For $\gamma\in\Gamma$ put $\iota_\gamma(\tau)=\gamma\tau\gamma^{-1}$. An additive map $f\colon R\to R$ is \emph{$\iota_\gamma$-graded} if $f(R_\tau)\subseteq R_{\iota_\gamma(\tau)}$ for all $\tau\in\Gamma$. If $\gamma$ is central, this is the usual notion of a graded map (compare \cite[Section 1.1]{Hazrat2016}).

\begin{definition}\label{YAMPXB0E}
Let $R$ be $\Gamma$-graded and $\gamma\in\Gamma$. A pair $(\sigma,\delta)$, where $\sigma$ is a ring endomorphism of $R$ and $\delta$ a $\sigma$-derivation, is \emph{$\gamma$-compatible} if
\begin{equation}\label{YAMB58LPXA}
    \sigma(R_\tau)\subseteq R_{\gamma\tau\gamma^{-1}}
    \quad \text{and}\quad
    \delta(R_\tau)\subseteq R_{\gamma\tau}
    \quad (\tau\in\Gamma).
\end{equation}
\end{definition}

Thus $\sigma$ is $\iota_\gamma$-graded, and $\delta$ multiplies degrees by $\gamma$ on the left. Our first result says that compatibility is exactly what is needed in general.

\begin{mainthm}\label{YAMK91QZWM}
Let $R$ be $\Gamma$-graded and $\gamma\in\Gamma$. The Ore extension $A=R[t;\sigma,\delta]$ admits a $\Gamma$-grading extending the grading of $R$ with $t\in A_\gamma$ if and only if $(\sigma,\delta)$ is $\gamma$-compatible. Such a grading is unique, and
\begin{equation}\label{YAMF65QNWT}
 A_\mu=\bigoplus_{n\ge0}R_{\mu\gamma^{-n}}t^n\quad (\mu\in\Gamma).
\end{equation}
\end{mainthm}

For abelian $\Gamma$, the conditions \eqref{YAMB58LPXA} say that $\sigma$ preserves degrees and $\delta$ shifts them by $\gamma$. For $\Gamma=\Z$ and $\gamma=1$ they define the usual graded Ore extensions, and for $\sigma=\id$ they recover the criterion of \cite[Propositions 2.1 and 2.2]{AitMohamed2026GradedDifferential}, where $\gamma$ must centralize $\Supp(R)$. For nonabelian $\Gamma$, the map $\sigma$ may move homogeneous components, and $\gamma$ need not centralize $\Supp(R)$ (\cref{YAML73VXTR}). This twist can be removed by a change of variable when $R_\gamma$ contains a unit (\cref{YAMN47ZXVL}(3)), but not in general (\cref{YAML73VXTR,YAMQWCZB}). We did not find this extension problem for nonabelian $\Gamma$ in the literature.

Next we ask when $A$ is gr-simple. An ideal $I$ of $R$ is \emph{$(\sigma,\delta)$-stable} if $\sigma(I)\subseteq I$ and $\delta(I)\subseteq I$, and a graded ring $R$ is \emph{$(\sigma,\delta)$-gr-simple} if $0$ and $R$ are its only $(\sigma,\delta)$-stable graded ideals. If $A$ is gr-simple, then $R$ is $(\sigma,\delta)$-gr-simple and $\delta$ is outer (\cref{YAMT64XWRP}). These are the graded forms of the standard necessary conditions for simplicity \cite[Proposition 4.6 and Remark 4.1]{OinertRichterSilvestrov2013}. For $\sigma=\id$ they are also sufficient in characteristic zero. This is due to Jordan for the trivial grading \cite{Jordan1975Thesis}, and it holds for arbitrary gradings by \cite[Theorem 1.10(2)]{AitMohamed2026GradedDifferential}. If $\sigma\neq\id$, sufficiency fails, already for the trivial grading: a simple Ore extension of a commutative domain has $\sigma=\id$ \cite[Theorem 4.4]{OinertRichterSilvestrov2013}, whereas the two conditions can hold with $\sigma\neq\id$. An example is the quantized Weyl algebra
\[
A=\Q\langle x,t\rangle/(tx-2xt-1),\quad \deg x=1,\quad \deg t=-1.
\]
Here $A$ is the Ore extension $\Q[x][t;\sigma,\delta]$, the ring $\Q[x]$ is $(\sigma,\delta)$-simple and $\delta$ is outer, but $A$ is not gr-simple (\cref{YAMQ95LNVK}). The obstruction is the homogeneous normal element $1+xt=tx-xt$, the classical normal element of the quantized Weyl algebra \cite{Jordan1995QuantizedWeyl}. Such an element can exist because commutators with $t$ need not lower the $t$-degree: the $t^{n+1}$-coefficient of $tf-ft$ is $\sigma(a_n)-a_n$ for $f=a_nt^n+\cdots$ (\cref{YAMZ32PXMT}).

Our next result shows that, over a gr-simple coefficient ring, elements of this kind are the only obstruction. We call $f\in A$ a \emph{monic normal element} if $f$ is homogeneous, normal, and monic of positive $t$-degree.

\begin{mainthm}\label{YAMJ76MXRP}
Let $R$ be $\Gamma$-graded and gr-simple, let $(\sigma,\delta)$ be $\gamma$-compatible with $\sigma$ an automorphism, and let $A=R[t;\sigma,\delta]$.
\begin{enumerate}[label=\textup{(\arabic*)}]
\item $A$ is gr-simple if and only if $A$ contains no monic normal element.
\item If $\Delta_n$ is outer for every $n\ge1$, then $A$ is gr-simple.
\item If $\operatorname{char}R=0$ and $\sigma\delta\sigma^{-1}-\delta$ is inner, then $A$ is gr-simple if and only if $\delta$ is outer.
\end{enumerate}
\end{mainthm}

In the quantized Weyl example, the coefficient ring $\Q[x]$ is $(\sigma,\delta)$-simple but not gr-simple, and the normal element $1+xt$ is not monic. The hypothesis that $R$ itself be gr-simple allows the leading coefficient to be normalized in the proof of \cref{YAMJ76MXRP}.

Part (3) applies in particular when $\sigma\delta=\delta\sigma$. For the trivial grading these statements are close to known results. Cauchon studied ideals of Ore extensions over quasi-simple rings \cite{Cauchon1979}. For $\sigma=\id$, the argument of part (1) is that of \cite[Proposition 4.14]{OinertRichterSilvestrov2013}. The argument of part (2) appears in \cite[Lemma 4.8]{eisenschlos2013}. Part (3) is related to a theorem of Chuang and Tsai on semi-invariant polynomials over prime rings \cite{ChuangTsai2009} (\cref{YAMR58CTLM}). Normal and semi-invariant polynomials are also studied in \cite{CortesFerrero2004,Jordan2002Normal,LamLeungLeroyMatczuk1989}. The graded statements are not covered by these results, because a gr-simple ring need not be prime: the group algebra $kC_2$, with $C_2=\langle g\rangle$, is gr-simple for its natural grading, but $(1+g)(1-g)=0$. We do not know whether the innerness hypothesis in part (3) can be removed (\cref{YAMR27NECS}). In positive characteristic the derivation $\frac{d}{dx}$ of $k[x^{\pm1}]$ already gives a ring that is not gr-simple, although all the conditions of part (3) except $\operatorname{char}R=0$ hold (\cref{YAMDSO0Q}).

Group algebras give a complete picture. Let $G$ be a group. A \emph{character} is a group homomorphism $\chi\colon G\to k^\times$, and $k_\chi$ denotes $k$ as a $G$-module via $g\cdot\lambda=\chi(g)\lambda$. We refer to \cite{Brown1982Cohomology} for group cohomology.

\begin{definition}\label{YAMZNTWS}
A \emph{crossed homomorphism} (or \emph{$1$-cocycle}) for $\chi$ is a map $c\colon G\to k$ with
\[
 c(\tau\rho)=c(\tau)+\chi(\tau)c(\rho)\quad(\tau,\rho\in G).
\]
The crossed homomorphisms form a $k$-vector space $Z^1(G,k_\chi)$. The \emph{coboundaries} are the cocycles $x(\chi-1)$ with $x\in k$. They form a subspace $B^1(G,k_\chi)$, and
\[
 H^1(G,k_\chi)=Z^1(G,k_\chi)/B^1(G,k_\chi).
\]
For $\chi=1$ this is $H^1(G,k)$ with trivial coefficients, and $Z^1(G,k_1)=\operatorname{Hom}(G,k)$.
\end{definition}

The \emph{affine group of the line} is the group $\mathrm{Aff}(1,k)\cong k\rtimes k^\times$ of maps $X\mapsto aX+b$ with $a\in k^\times$ and $b\in k$, under composition. A group is \emph{metabelian} if it has an abelian normal subgroup with abelian quotient. Every subgroup of $\mathrm{Aff}(1,k)$ is metabelian. If $G$ acts on a ring $S$ by automorphisms $\theta_\tau$, the \emph{skew group ring} $S\rtimes_\theta G=\bigoplus_{\tau\in G}Su_\tau$ has multiplication $(fu_\tau)(gu_\rho)=f\theta_\tau(g)u_{\tau\rho}$, and it is $G$-graded by $\deg u_\tau=\tau$. Grade $kG$ by $G$, fix $\gamma\in G$ and a character $\chi$, and put $\sigma(\tau)=\chi(\tau)\gamma\tau\gamma^{-1}$. Every $\iota_\gamma$-graded $k$-algebra endomorphism of $kG$ has this form for some $\chi$. The $k$-linear $\sigma$-derivations $\delta$ for which $(\sigma,\delta)$ is $\gamma$-compatible are exactly the maps $\delta_c(\tau)=c(\tau)\gamma\tau$ with $c\in Z^1(G,k_\chi)$, and $\delta_c$ is inner if and only if $c\in B^1(G,k_\chi)$ (\cref{YAMM18QZKN}). The cocycle identity for $c$ says precisely that
\[
 \Phi\colon G\to\mathrm{Aff}(1,k),\quad \Phi(\tau)=\varphi_\tau,\quad \varphi_\tau(X)=\chi(\tau)X+c(\tau),
\]
is a group homomorphism, and every homomorphism $G\to\mathrm{Aff}(1,k)$ arises in this way.

\begin{mainthm}\label{YAMD53GRPA}
Let $A=kG[t;\sigma,\delta_c]$, graded by $\deg t=\gamma$, and let $\bar k$ be an algebraic closure of $k$. Let $G$ act on $k[X]$ by $\theta_\tau(f)=f\circ\varphi_\tau^{-1}$.
\begin{enumerate}[label=\textup{(\arabic*)}]
\item The map $f(X)u_\tau\mapsto f(\gamma^{-1}t)\tau$ is an isomorphism of $G$-graded rings $k[X]\rtimes_\theta G\to A$.
\item The graded ideals of $A$ are $0$ and the ideals $Af(\gamma^{-1}t)$, where $f\in k[X]$ is monic with $f\circ\varphi_\tau\in k^\times f$ for all $\tau\in G$. Distinct $f$ give distinct ideals.
\item $A$ is gr-simple if and only if $\Phi(G)$ has no finite orbit in $\bar k$, if and only if the restriction of $[c]$ to every subgroup of finite index is nonzero. If $\operatorname{char}k=0$, this holds if and only if $[c]\neq0$ in $H^1(G,k_\chi)$.
\item $A$ is simple if and only if it is gr-simple and $\Phi$ is injective. In particular, if $A$ is simple, then $G$ is metabelian.
\end{enumerate}
\end{mainthm}

By part (1), the ring $A$ depends on $(\chi,c)$ but not on $\gamma$. Part (2) is the correspondence between graded ideals of a strongly graded ring and invariant ideals of its identity component \cite[Proposition 2.11.7]{NastasescuVanOystaeyen2004}, made explicit, and part (4) rests on a theorem of \"Oinert on skew group rings \cite{Oinert2009}. See \cite{Oinert2014} for abelian groups. Ore extensions of group algebras also arise as Hopf--Ore extensions \cite{BrownOHaganZhangZhuang2015,Panov2003}, where the Hopf structure imposes further conditions. We use no Hopf structure.

In positive characteristic the cohomological criterion fails. If $G$ is finite, then $A$ is never gr-simple, although $H^1(G,k_\chi)$ may be nonzero, for instance for $G=C_p$ in characteristic $p$ (\cref{YAMC19FING}). For $G=S_3\cong\mathrm{Aff}(1,\mathbb F_3)$ the obstruction is carried by an orbit rather than by a fixed point (\cref{YAME35FTHR} and \cref{YAMR46ORBT}).

Parts (3) and (4) of \cref{YAMD53GRPA} show that gr-simplicity does not imply simplicity. For the free group $F(a,b)$, $\gamma=a$ and $q\in\Q\setminus\{0,1\}$, the ring with $ta=q\,at$ and $tb=aba^{-1}t+ab$ is gr-simple but not simple (\cref{YAMK82RTXW,YAME72BSLM}), and for $G=\Z^2$ simplicity depends on the arithmetic of the values of $c$ (\cref{YAMQJ76P}). In all these examples $R_\gamma$ contains a unit, so the degree of $t$ can be moved to $e$ by the change of variable $s=\gamma^{-1}t$. The twist can also coexist with gr-simplicity. In \cref{YAMQWCZB}, $\Gamma$ is the fundamental group of the Klein bottle, $\sigma$ is not graded, $R_\gamma=0$, and $A$ is gr-simple but not simple. \cref{YAML73VXTR} also has an intrinsic twist, but it is not gr-simple.

The same questions make sense for Poisson algebras. Let $P$ be a Poisson algebra over $k$, let $\alpha$ be a Poisson derivation of $P$ and $\delta$ a Poisson $\alpha$-derivation. By a theorem of Oh \cite{Oh2006}, in the sign convention of \cite{LaunoisLecoutre2016}, $P[t]$ has a unique Poisson bracket that extends the bracket of $P$ and satisfies $\{t,a\}=\alpha(a)t+\delta(a)$. This is the \emph{Poisson polynomial ring} $B=P[t;\alpha,\delta]_{\mathrm p}$. The derivation $\delta$ is \emph{$\alpha$-inner} if $\delta(a)=\{u,a\}-u\alpha(a)$ for some $u\in P$, and an element $f\in B$ is \emph{Poisson normal} if $\{B,f\}\subseteq fB$. In a commutative graded ring, nonzero products and brackets involve only commuting degrees, so the twist $\iota_\gamma$ of \cref{YAMK91QZWM} disappears, and we take $\Gamma$ abelian. The twist also disappears in semiclassical limits (\cref{YAMR19NOGO}). A pair $(\alpha,\delta)$ is \emph{$\gamma$-compatible} if $\alpha(P_\tau)\subseteq P_\tau$ and $\delta(P_\tau)\subseteq P_{\gamma\tau}$ for all $\tau$, and $P$ is \emph{Poisson gr-simple} if $0$ and $P$ are its only graded Poisson ideals. The precise definitions are recalled in \cref{YAMP55QSCL}.

\begin{mainthm}\label{YAME61POIS}
Let $\Gamma$ be abelian, $P$ a $\Gamma$-graded Poisson algebra, $\gamma\in\Gamma$ and $B=P[t;\alpha,\delta]_{\mathrm p}$.
\begin{enumerate}[label=\textup{(\arabic*)}]
\item $B$ admits a grading extending that of $P$ with $t\in B_\gamma$ if and only if $(\alpha,\delta)$ is $\gamma$-compatible. Such a grading is unique.
\item Assume that this holds and that $P$ is Poisson gr-simple. Then $B$ is Poisson gr-simple if and only if $B$ contains no homogeneous Poisson normal element $t^n+\sum_{i<n}b_it^i$ with $n\ge1$.
\item If moreover $\operatorname{char}k=0$, then $B$ is Poisson gr-simple if and only if $\delta$ is not $\alpha$-inner.
\end{enumerate}
\end{mainthm}

Part (3) needs no analogue of the innerness hypothesis of \cref{YAMJ76MXRP}(3). The reason is that the $t^{n-1}$-coefficient of $\{a,t^n\}$ is $-n\delta(a)$, whereas in the skew case the corresponding coefficient involves $\Delta_n$ (\cref{YAMG15WRKP}(2)). Simplicity criteria for Poisson Ore extensions based on Poisson normal elements also appear in \cite{OhSim2021}. Semiclassical limits connect the two settings. The quantized Weyl algebra above lies in the graded family $zx=qxz+(q-1)$, whose semiclassical limit is the Poisson polynomial ring with $\{z,x\}=xz+1$. The element $1+xz$ is normal in the family and Poisson normal in the limit (\cref{YAME58QWSL}). Moreover, if $\delta=\delta_u$ is inner, then its semiclassical limit is $\alpha$-inner (\cref{YAML73SCLV}).

The paper is organized as follows. In \cref{YAMD14WQZK} we prove \cref{YAMK91QZWM}, study $\gamma$-compatible pairs and construct the nonabelian examples. \cref{YAMH82TVLX} is devoted to gr-simplicity and proves \cref{YAMJ76MXRP,YAMD53GRPA}. \cref{YAMP55QSCL} treats Poisson polynomial rings and semiclassical limits, and proves \cref{YAME61POIS}.

\section{Gradings on Ore extensions}\label{YAMD14WQZK}

Throughout this section, $R$ is a $\Gamma$-graded ring, $\sigma$ is a ring endomorphism of $R$, $\delta$ is a $\sigma$-derivation, and $A=R[t;\sigma,\delta]$. We first prove \cref{YAMK91QZWM}.

\begin{proof}[\textnormal{\textbf{Proof of \cref{YAMK91QZWM}}}]
Suppose that $(\sigma,\delta)$ is $\gamma$-compatible, and define $A_\mu$ by \eqref{YAMF65QNWT}. Since $A=\bigoplus_nRt^n$ is free over $R$ and $R=\bigoplus_\rho R_\rho$, we have $A=\bigoplus_{n,\rho}R_\rho t^n$. Placing $R_\rho t^n$ in $A_{\rho\gamma^n}$ gives $A=\bigoplus_\mu A_\mu$, and $1\in A_e$. It remains to prove $A_\mu A_\beta\subseteq A_{\mu\beta}$.

First, $R_\rho A_\beta\subseteq A_{\rho\beta}$, because $R_\rho R_{\beta\gamma^{-n}}\subseteq R_{\rho\beta\gamma^{-n}}$. Next, $tA_\beta\subseteq A_{\gamma\beta}$: for $r\in R_\rho$ and $n\ge0$ we have $rt^n\in A_{\rho\gamma^n}$, and by \eqref{YAMO47UXRT} and \eqref{YAMB58LPXA},
\begin{align*}
 t\,rt^n&=\sigma(r)t^{n+1}+\delta(r)t^n\\
 &\in R_{\gamma\rho\gamma^{-1}}t^{n+1}+R_{\gamma\rho}t^n\subseteq A_{\gamma\rho\gamma^n}.
\end{align*}
By induction on $m$, $t^mA_\beta\subseteq A_{\gamma^m\beta}$ for all $m\ge0$: if it holds for $m$, then
\begin{align*}
 t^{m+1}A_\beta\subseteq tA_{\gamma^m\beta}\subseteq A_{\gamma^{m+1}\beta}.
\end{align*}
Hence, for $r\in R_\rho$,
\begin{align*}
 (rt^m)A_\beta\subseteq R_\rho A_{\gamma^m\beta}\subseteq A_{\rho\gamma^m\beta}.
\end{align*} Since $A_\mu$ is spanned by such $rt^m$ with $\rho\gamma^m=\mu$, the grading is multiplicative.

Conversely, let $\{A'_\mu\}$ be a $\Gamma$-grading of $A$ with $R_\tau\subseteq A'_\tau$ for all $\tau$ and $t\in A'_\gamma$. Let $0\neq r\in R_\tau$, and write $\sigma(r)=\sum_\rho s_\rho$ and $\delta(r)=\sum_\rho d_\rho$ with $s_\rho,d_\rho\in R_\rho$. Then $s_\rho t\in A'_{\rho\gamma}$, $d_\rho\in A'_\rho$ and $tr\in A'_{\gamma\tau}$. For $\beta\neq\gamma\tau$, the component of $tr=\sum_\rho s_\rho t+\sum_\rho d_\rho$ of degree $\beta$ is $s_{\beta\gamma^{-1}}t+d_\beta$, and it vanishes. Since $1,t$ are $R$-linearly independent, $s_{\beta\gamma^{-1}}=0=d_\beta$. Hence $\sigma(r)\in R_{\gamma\tau\gamma^{-1}}$ and $\delta(r)\in R_{\gamma\tau}$, so $(\sigma,\delta)$ is $\gamma$-compatible and the grading $\{A_\mu\}$ of \eqref{YAMF65QNWT} is defined. Since $t^n\in A'_{\gamma^n}$, we have $R_\rho t^n\subseteq A'_{\rho\gamma^n}$, so $A_\mu\subseteq A'_\mu$ for every $\mu$. If $a\in A'_\mu$, write $a=\sum_\beta a_\beta$ with $a_\beta\in A_\beta\subseteq A'_\beta$. Uniqueness of homogeneous components in $\{A'_\beta\}$ gives $a=a_\mu\in A_\mu$. Hence $A'_\mu=A_\mu$.
\end{proof}

If $\delta\neq0$, then $\gamma$ is determined by $\delta$. Indeed, choose a homogeneous $x\in R_\tau$ with $\delta(x)\neq0$. If $\gamma$ and $\gamma'$ both satisfy \eqref{YAMB58LPXA}, then $0\neq\delta(x)\in R_{\gamma\tau}\cap R_{\gamma'\tau}$, so $\gamma=\gamma'$.

\begin{remark}\label{YAMN47ZXVL}
(1) For nonabelian $\Gamma$, $\sigma$ is $\iota_\gamma$-graded but need not be graded: $\sigma(R_\tau)\subseteq R_\tau$ for all $\tau$ if and only if $\gamma\tau\gamma^{-1}=\tau$ whenever $\sigma(R_\tau)\neq0$. In \cref{YAML73VXTR,YAMQWCZB,YAMK82RTXW}, $\sigma$ is not graded.

(2) If $\sigma=\id$, then $R_\tau\subseteq R_{\gamma\tau\gamma^{-1}}$ forces $\gamma\tau\gamma^{-1}=\tau$ for $\tau\in\Supp(R)$. So $\gamma$ centralizes $\Supp(R)$ and $\delta(R_\tau)\subseteq R_{\gamma\tau}=R_{\tau\gamma}$. This recovers \cite[Proposition 2.1]{AitMohamed2026GradedDifferential}.

(3) Let $(\sigma,\delta)$ be $\gamma$-compatible and let $g\in R_\gamma$ be a unit. Then $g^{-1}\in R_{\gamma^{-1}}$ (compare components of $gg^{-1}=1$). Put $t'=g^{-1}t\in A_e$, $\sigma'(r)=g^{-1}\sigma(r)g$ and $\delta'=g^{-1}\delta$. Then $t'r=\sigma'(r)t'+\delta'(r)$, $\sigma'$ is a ring endomorphism, and
\begin{align*}
 \delta'(rr')&=g^{-1}\bigl(\sigma(r)\delta(r')+\delta(r)r'\bigr)\\
 &=\sigma'(r)\delta'(r')+\delta'(r)r',
\end{align*}
so $\delta'$ is a $\sigma'$-derivation. Moreover,
\begin{align*}
 \sigma'(R_\tau)&\subseteq R_{\gamma^{-1}}R_{\gamma\tau\gamma^{-1}}R_\gamma\subseteq R_\tau,\\
 \delta'(R_\tau)&\subseteq R_{\gamma^{-1}}R_{\gamma\tau}\subseteq R_\tau.
\end{align*} Since $t'^{n+1}=g^{-1}t\,t'^n$, induction gives $t'^n=u_nt^n+(\text{lower terms})$ with $u_{n+1}=g^{-1}\sigma(u_n)$, that is,
\[
u_n=g^{-1}\sigma(g^{-1})\cdots\sigma^{n-1}(g^{-1}),
\]
a unit. So $\{t'^n\}_{n\ge0}$ is a left $R$-basis of $A$, and $A=R[t';\sigma',\delta']$ with $\deg t'=e$ and $\sigma'$ graded. In particular, if $\sigma(r)=grg^{-1}$ for all $r$, then $\sigma'=\id$ and $A$ is the graded differential polynomial ring $R[t';\delta']$. Thus, if the degree of $t$ cannot be moved to $e$ in this way, then $R_\gamma$ contains no unit (\cref{YAML73VXTR,YAMQWCZB}).
\end{remark}

\begin{lemma}\label{YAMW28KQRP}
Let $(\sigma,\delta)$ be $\gamma$-compatible.
\begin{enumerate}[label=\textup{(\arabic*)}]
\item For $u\in R_\gamma$, the pair $(\sigma,\delta_u)$ is $\gamma$-compatible.
\item If $\delta=\delta_u$ with $u=\sum_\rho u_\rho$ and $u_\rho\in R_\rho$, then $\delta=\delta_{u_\gamma}$. In particular, $\delta$ is inner if and only if $\delta=\delta_u$ for some $u\in R_\gamma$.
\item If $\sigma$ is an automorphism, then $\sigma^j(R_\tau)=R_{\gamma^j\tau\gamma^{-j}}$ for all $j\in\Z$ and $\tau\in\Gamma$. Each $\sigma^j\delta\sigma^{-j}$ is a $\sigma$-derivation, and the pair $(\sigma,\sigma^j\delta\sigma^{-j})$ is $\gamma$-compatible for every $j\in\Z$. Moreover, $\sigma\delta_u\sigma^{-1}=\delta_{\sigma(u)}$, and $\sigma(u)\in R_\gamma$ for $u\in R_\gamma$.
\end{enumerate}
\end{lemma}

\begin{proof}
(1) For $r\in R_\tau$ we have $ur\in R_{\gamma\tau}$ and $\sigma(r)u\in R_{\gamma\tau\gamma^{-1}}R_\gamma\subseteq R_{\gamma\tau}$, so $\delta_u(r)\in R_{\gamma\tau}$.

(2) Let $r\in R_\tau$. Then $u_\rho r\in R_{\rho\tau}$ and $\sigma(r)u_\rho\in R_{\gamma\tau\gamma^{-1}\rho}$. Each of these degrees equals $\gamma\tau$ only for $\rho=\gamma$. Since $\delta(r)\in R_{\gamma\tau}$, comparing components of degree $\gamma\tau$ in $\delta(r)=\sum_\rho(u_\rho r-\sigma(r)u_\rho)$ gives $\delta(r)=u_\gamma r-\sigma(r)u_\gamma$. By additivity, $\delta=\delta_{u_\gamma}$.

(3) We first show $\sigma(R_\tau)=R_{\gamma\tau\gamma^{-1}}$. The inclusion $\subseteq$ is compatibility. Let $x\in R_{\gamma\tau\gamma^{-1}}$ and write $x=\sigma(y)$ with $y=\sum_\rho y_\rho$. The terms $\sigma(y_\rho)\in R_{\gamma\rho\gamma^{-1}}$ lie in pairwise distinct degrees, so $x=\sigma(y_\tau)$. Hence $\sigma(R_\tau)=R_{\gamma\tau\gamma^{-1}}$ and, $\sigma$ being bijective, $\sigma^{-1}(R_\rho)=R_{\gamma^{-1}\rho\gamma}$. Induction on $|j|$ gives $\sigma^j(R_\tau)=R_{\gamma^j\tau\gamma^{-j}}$ for all $j\in\Z$.

Put $\delta_j=\sigma^j\delta\sigma^{-j}$. Since $\sigma^{\pm j}$ are ring homomorphisms,
\begin{align*}
 \delta_j(ab)&=\sigma^j\bigl(\sigma^{1-j}(a)\,\delta(\sigma^{-j}(b))+\delta(\sigma^{-j}(a))\,\sigma^{-j}(b)\bigr)\\
 &=\sigma(a)\delta_j(b)+\delta_j(a)b,
\end{align*}
so $\delta_j$ is a $\sigma$-derivation. For $a\in R_\tau$ we get $\sigma^{-j}(a)\in R_{\gamma^{-j}\tau\gamma^{j}}$, then $\delta\sigma^{-j}(a)\in R_{\gamma^{1-j}\tau\gamma^{j}}$, and then
\begin{align*}
 \delta_j(a)\in R_{\gamma^{j}\gamma^{1-j}\tau\gamma^{j}\gamma^{-j}}=R_{\gamma\tau}.
\end{align*}
Finally, applying $\sigma$ to $\delta_u(\sigma^{-1}(r))=u\sigma^{-1}(r)-ru$ gives
\begin{align*}
 \sigma\delta_u\sigma^{-1}(r)=\sigma(u)r-\sigma(r)\sigma(u)=\delta_{\sigma(u)}(r),
\end{align*}
and $\sigma(u)\in R_{\gamma\gamma\gamma^{-1}}=R_\gamma$.
\end{proof}

As in the ungraded case \cite[Lemma 1.5]{Goodearl1992}, the inner case reduces to a skew polynomial ring.

\begin{lemma}\label{YAMS91TXMK}
Let $(\sigma,\delta)$ be $\gamma$-compatible with $\delta=\delta_u$, $u\in R_\gamma$. Then $s=t-u\in A_\gamma$ satisfies $sr=\sigma(r)s$ for $r\in R$, $A=\bigoplus_{n\ge0}Rs^n$ is the graded skew polynomial ring $R[s;\sigma]$ with $\deg s=\gamma$, and $As$ is a nonzero proper graded ideal of $A$.
\end{lemma}

\begin{proof}
Since $\delta(r)=ur-\sigma(r)u$,
\[
 sr=\sigma(r)t+\delta(r)-ur=\sigma(r)t-\sigma(r)u=\sigma(r)s.
\]
By \eqref{YAMO47UXRT}, $t^iu\in\sigma^i(u)t^i+\sum_{m<i}Rt^m$, and induction on $n$ gives $s^n-t^n\in\sum_{m<n}Rt^m$. Hence $\{s^n\}_{n\ge0}$ is a left $R$-basis of $A$. Since $s(rs^n)=\sigma(r)s^{n+1}\in As$, we get $sA\subseteq As$, so the left ideal $As$ is two-sided. In the basis $\{s^n\}$ it equals $\bigoplus_{n\ge1}Rs^n$, so $1\notin As$. It is graded because $s$ is homogeneous.
\end{proof}

\begin{example}[Quantized Weyl algebra]\label{YAMC56NPQW}
Let $R=\Q[x]$ with $\deg x=1$, $\sigma(f)(x)=f(2x)$ and $\delta(f)=(f(2x)-f(x))/x$. Then $\delta(fg)=\sigma(f)\delta(g)+\delta(f)g$, $\sigma$ preserves degrees, and
\[
\delta(x^m)=(2^m-1)x^{m-1}\quad(m\ge1).
\]
Thus $(\sigma,\delta)$ is $(-1)$-compatible. Since $tx=2xt+1$, we have
\[
 A=\Q[x][t;\sigma,\delta]\cong\Q\langle x,t\rangle/(tx-2xt-1),
\]
and by \cref{YAMK91QZWM} the ring $A$ is graded by $\deg x=1$ and $\deg t=-1$.
\end{example}

In the next example $\sigma$ moves homogeneous components and $t$ has noncentral degree.

In the following diagram, the vertices are the degrees of $x_1,x_2,x_3$ in \cref{YAML73VXTR}, with $\gamma=(123)$. Solid arrows represent $\tau\mapsto\gamma\tau\gamma^{-1}$, the effect of $\sigma$ on degrees, whereas dashed arrows represent $\tau\mapsto\gamma\tau$, the effect of $\delta$. These maps rotate the three transpositions in opposite directions, and $\gamma$ commutes with none of them. Permutations are composed from right to left.

\begin{figure}[ht]
\centering
\begin{tikzpicture}[>={Stealth[length=2.2mm]}, every node/.style={font=\small}]
\node (a) at (90:1.7) {$\deg x_1=(12)$};
\node (b) at (210:1.7) {$\deg x_2=(23)$};
\node (c) at (330:1.7) {$\deg x_3=(13)$};
\draw[->, thick] (a) to[bend right=25] node[left=2pt]{$\iota_\gamma$} (b);
\draw[->, thick] (b) to[bend right=25] (c);
\draw[->, thick] (c) to[bend right=25] (a);
\draw[->, dashed] (a) to[bend right=12] node[pos=0.55, left=1pt]{$\gamma\cdot$} (c);
\draw[->, dashed] (c) to[bend right=12] (b);
\draw[->, dashed] (b) to[bend right=12] (a);
\end{tikzpicture}
\end{figure}

\begin{example}[A free algebra graded by $S_3$]\label{YAML73VXTR}
Compose permutations from right to left. Let $\Gamma=S_3$, $\gamma=(123)$, and let $R=k\langle x_1,x_2,x_3\rangle$ be the free algebra graded by
\[
 \deg x_1=(12),\quad \deg x_2=(23),\quad \deg x_3=(13).
\]
Then $\gamma(12)\gamma^{-1}=(23)$, $\gamma(23)\gamma^{-1}=(13)$, $\gamma(13)\gamma^{-1}=(12)$, and $\gamma(12)=(13)$, $\gamma(23)=(12)$, $\gamma(13)=(23)$. Let $\sigma$ be the automorphism $x_1\mapsto x_2\mapsto x_3\mapsto x_1$, and let $\delta$ be the $\sigma$-derivation with
\[
 \delta(x_1)=x_3,\quad \delta(x_2)=x_1,\quad \delta(x_3)=x_2.
\]
It exists and is unique because $R$ is free. Thus $\deg\sigma(x_i)=\gamma\deg(x_i)\gamma^{-1}$ and $\deg\delta(x_i)=\gamma\deg(x_i)$. Both properties pass to products: if $u\in R_\mu$ and $v\in R_\beta$ satisfy them, then $\sigma(uv)\in R_{\gamma\mu\beta\gamma^{-1}}$ and $\delta(uv)=\sigma(u)\delta(v)+\delta(u)v\in R_{\gamma\mu\beta}$. Hence $(\sigma,\delta)$ is $\gamma$-compatible, and $A=R[t;\sigma,\delta]$ is graded with $\deg t=(123)$.

The transpositions generate $S_3$, whose centre is trivial, so no $\gamma\neq e$ centralizes $\Supp(R)$. By \cref{YAMN47ZXVL}(2), a differential polynomial ring over $R$ admits a compatible grading only with $\deg t=e$. The homogeneous units of $R$ are the nonzero scalars, so $R_\gamma$ contains no unit and \cref{YAMN47ZXVL}(3) does not apply.

The derivation $\delta$ is outer. The degree of a monomial of length $\ell$ is a permutation of sign $(-1)^\ell$, and $\gamma$ is even, so every monomial in $R_\gamma$ has even length $\ell\ge2$. For $u\in R_\gamma$, every monomial of $\delta_u(x_1)=ux_1-x_2u$ therefore has length at least $3$, whereas $\delta(x_1)=x_3$. Hence $\delta\neq\delta_u$ for $u\in R_\gamma$, so $\delta$ is outer by \cref{YAMW28KQRP}(2). On the other hand, the ideal generated by $x_1,x_2,x_3$ is graded and stable under $\sigma$ and $\delta$, so $A$ is not gr-simple by \cref{YAMT64XWRP} below.

The degree of $t$ cannot be moved to $e$: no $s\in A_e$ has powers forming a left $R$-basis of $A$. Suppose that $s\in A_e$ is such an element. Since $R$ is a domain and $\sigma$ is injective, $\lc(fg)=\lc(f)\sigma^{\degt f}(\lc(g))\neq0$, so $\degt(fg)=\degt f+\degt g$ for nonzero $f,g\in A$. Then $d=\degt s\ge1$ because $A\neq R$, and every nonzero element of $\bigoplus_nRs^n$ has $t$-degree divisible by $d$. Applied to $t$, this gives $d=1$. Write $s=bt+a$ and $t=c_1s+c_0$ with $a,b,c_0,c_1\in R$. By \eqref{YAMF65QNWT}, $b\in R_{\gamma^{-1}}$, and comparing $t$-coefficients gives $c_1b=1$. As $R$ is a domain, $(bc_1-1)b=0$ gives $bc_1=1$, so $b$ is a unit of the free algebra, hence a nonzero scalar, and $b\in R_e$. This contradicts $b\in R_{\gamma^{-1}}$ with $\gamma^{-1}\neq e$.
\end{example}

The next example shows that the twist can be intrinsic in a gr-simple extension.

\begin{example}[A gr-simple extension with an intrinsic twist]\label{YAMQWCZB}
Let $\Gamma=\langle h,r\mid rhr^{-1}=h^{-1}\rangle$ be the fundamental group of the Klein bottle. Every element of $\Gamma$ is uniquely $h^mr^n$ with $m,n\in\Z$. Let $R=\Q[h^{\pm1}][y;\theta]$ with $\theta(h)=h^{-1}$, graded by $\deg h=h$ and $\deg y=r$. Thus $\{h^my^i\}$ is a $\Q$-basis of $R$, $h^my^i\in R_{h^mr^i}$, and
\[
(h^my^i)(h^ny^j)=h^{m+(-1)^in}y^{i+j}.
\]
Put $\gamma=r^{-1}$ and define $\Q$-linear maps by
\begin{align*}
 \sigma(h^my^i)=h^{-m}y^i,\qquad \delta(h^m)=0,\qquad \delta(h^my^i)=i\,h^{-m}y^{i-1}\quad(i\ge1).
\end{align*}
Both sides of $\sigma(ab)=\sigma(a)\sigma(b)$ equal $h^{-m-(-1)^in}y^{i+j}$ for $a=h^my^i$ and $b=h^ny^j$, so $\sigma$ is an automorphism with $\sigma^2=\id$. Since $(-1)^{i-1}=-(-1)^i$,
\begin{align*}
 \sigma(a)\delta(b)+\delta(a)b&=j\,h^{-m-(-1)^in}y^{i+j-1}+i\,h^{-m+(-1)^{i-1}n}y^{i+j-1}\\
 &=\delta(ab),
\end{align*}
so $\delta$ is a $\sigma$-derivation. As $\gamma h^mr^i\gamma^{-1}=h^{-m}r^i$ and $\gamma h^mr^i=h^{-m}r^{i-1}$, the pair $(\sigma,\delta)$ is $\gamma$-compatible, and $A=R[t;\sigma,\delta]$ is graded with $\deg t=r^{-1}$. Its defining relations are
\[
yh=h^{-1}y,\quad th=h^{-1}t,\quad ty=yt+1.
\]
Here $\gamma$ does not centralize $h\in\Supp(R)$, so $\sigma$ is not graded, and $R_\gamma=0$ because $\Supp(R)=\{h^mr^n:m\in\Z,\ n\ge0\}$.

The ring $A$ is gr-simple. Let $W$ be the subalgebra generated by $y$ and $t$. The elements $h^my^it^j$ form a $\Q$-basis of $A$, so $W\cong A_1(\Q)$, and $h^my^it^j\in A_{h^mr^{i-j}}$ gives $A_{h^mr^n}=h^mW_n$ with $W_n=\operatorname{span}\{y^it^j:i-j=n\}$. Let $J\neq0$ be a graded ideal and $0\neq f\in J_{h^mr^n}$. Then $h^{-m}f\in J\cap W$, so $J\cap W$ is a nonzero ideal of the simple ring $A_1(\Q)$. Hence $1\in J$.

The ring $A$ is not simple. The element $h+h^{-1}$ commutes with $h$, $y$ and $t$, so it is central. Since $R$ is a domain and $\sigma$ is injective, the $t$-degree is additive on $A$ (see \cref{YAML73VXTR}), so the units of $A$ lie in $R$. Likewise, the units of $R$ lie in $\Q[h^{\pm1}]$ and are the elements $\lambda h^m$ with $\lambda\in\Q^\times$. Hence $h+h^{-1}-2$ is a central nonunit, and $A(h+h^{-1}-2)$ is a nonzero proper ideal.

The degree of $t$ cannot be moved to $e$. By \eqref{YAMF65QNWT}, $A_e=\bigoplus_n\Q\,y^nt^n$. The first part of the argument of \cref{YAML73VXTR} shows that an element $s\in A_e$ whose powers form a left $R$-basis has $t$-degree one. So $s=b\,yt+a$ with $a,b\in\Q$ and $b\neq0$, and $t=c_1s+c_0$ with $c_0,c_1\in R$ forces $c_1by=1$ in $R$, which is impossible since $y$ is not a unit.

Note that $R$ is not gr-simple, since $Ry=yR$ is a proper graded ideal, so \cref{YAMJ76MXRP} does not apply.
\end{example}

\section{Gr-simplicity}\label{YAMH82TVLX}

Throughout this section $(\sigma,\delta)$ is $\gamma$-compatible and $A$ carries the grading \eqref{YAMF65QNWT}.

\subsection{Necessary conditions and a counterexample}

\begin{proposition}\label{YAMT64XWRP}
If $A$ is gr-simple, then $R$ is $(\sigma,\delta)$-gr-simple and $\delta$ is outer.
\end{proposition}

\begin{proof}
Let $I$ be a graded ideal of $R$ with $\sigma(I)+\delta(I)\subseteq I$, and put $IA=\bigoplus_nIt^n$. Clearly $R(IA)\subseteq IA$ and $(IA)t\subseteq IA$. Iterating \eqref{YAMO47UXRT} gives $t^nr\in\sum_{m\le n}Rt^m$, so $(it^n)r\in\sum_mIt^m$ for $i\in I$ and $r\in R$. Finally $t(it^n)=\sigma(i)t^{n+1}+\delta(i)t^n\in IA$. Since $R$ and $t$ generate $A$, $IA$ is an ideal of $A$. It is graded by \eqref{YAMF65QNWT}, and $IA\cap R=I$. Hence $I\in\{0,R\}$ when $A$ is gr-simple. If $\delta$ is inner, then $\delta=\delta_u$ with $u\in R_\gamma$ by \cref{YAMW28KQRP}(2), and \cref{YAMS91TXMK} gives a nonzero proper graded ideal.
\end{proof}

For $\sigma=\id$ the converse holds in characteristic zero \cite[Theorem 1.10(2)]{AitMohamed2026GradedDifferential}. For $\sigma\neq\id$ it fails.

\begin{proposition}\label{YAMQ95LNVK}
In \cref{YAMC56NPQW}, $R=\Q[x]$ is $(\sigma,\delta)$-simple and $\delta$ is outer, but $A$ is not gr-simple. In particular, $A$ is not simple.
\end{proposition}

\begin{proof}
Let $(g)$ be a nonzero $(\sigma,\delta)$-stable ideal of $\Q[x]$. Since $g(2x)\in(g)$ has the same degree as $g$, we have $g(2x)=\lambda g(x)$ with $\lambda\in\Q^\times$. If $a\in\bar\Q$ is a root of $g$, then $\lambda g(a/2)=g(a)=0$, so the finite set of roots of $g$ is stable under $a\mapsto a/2$. Hence $0$ is the only possible root and $g=\mu x^m$ with $\mu\in\Q^\times$. Since $\delta(x^m)=(2^m-1)x^{m-1}$, $\delta$-stability forces $m=0$, that is, $(g)=\Q[x]$. As $R_{-1}=0$, \cref{YAMW28KQRP}(2) shows that the only inner $\sigma$-derivation in this setting is $0$, while $\delta(x)=1$. So $\delta$ is outer.

Let $w=1+xt\in A_0$. From $tx=2xt+1$ we get
\[
wx=x+x(2xt+1)=2xw,\quad tw=t+(2xt+1)t=2wt.
\]
Since $2$ is invertible, $wA=Aw$, and $Aw$ is a nonzero graded ideal. The assignment $x\mapsto X$, $t\mapsto-X^{-1}$ kills $tx-2xt-1$, so it defines a homomorphism $A\to\Q[X^{\pm1}]$. It kills $w$ and is the identity on $\Q$, so $Aw\neq A$.
\end{proof}

\begin{remark}\label{YAMZ32PXMT}
The element $w$ comes from a localization. On the gr-simple ring $R'=\Q[x^{\pm1}]$ the maps $\sigma$ and $\delta$ extend by the same formulas, and $\delta=\delta_u$ with $u=-x^{-1}\in R'_{-1}$, since both are $\sigma$-derivations and $\delta_u(x)=-1+2=1=\delta(x)$. By \cref{YAMS91TXMK}, $s=t+x^{-1}$ generates a proper graded ideal of $R'[t;\sigma,\delta]$, and $w=xs$. Thus $\delta$ becomes inner after inverting $x$. The element $w$ is normal of $t$-degree one (see \cite{Jordan2002Normal} for such elements and \cite{Jordan1995QuantizedWeyl} for localizations of the quantized Weyl algebra). The relation $tx=2xt+1$ is also the $q$-oscillator relation $aa^\dagger-q\,a^\dagger a=1$ of Arik and Coon \cite{ArikCoon1976} with $q=2$, $a=t$ and $a^\dagger=x$ (see also \cite{KirkmanSmall1993}).

The leading-coefficient argument of \cite{AitMohamed2026GradedDifferential} fails because, for $f=a_nt^n+a_{n-1}t^{n-1}+\cdots$,
\[
 tf-ft=(\sigma(a_n)-a_n)t^{n+1}+\bigl(\delta(a_n)+\sigma(a_{n-1})-a_{n-1}\bigr)t^n+\cdots.
\]
The commutator with $t$ no longer lowers the $t$-degree, and ideals of leading coefficients need not be $\delta$-stable. Here the leading coefficients of the elements of $Aw$ of $t$-degree one, together with $0$, form the ideal $(x)$: if $a_1t+a_0\in Aw$, applying the homomorphism above gives $a_1=xa_0$, and conversely $a_0w=a_0+xa_0t$. The ideal $(x)$ is not $\delta$-stable.
\end{remark}

\subsection{A characterization over gr-simple coefficient rings}

From now on $\sigma$ is an automorphism. By \cref{YAMW28KQRP}(3), $\sigma(R_\tau)=R_{\iota_\gamma(\tau)}$ for every $\tau$. The proof of \cref{YAMJ76MXRP} rests on ideals of leading coefficients. For $\sigma=\id$ and the trivial grading, compare \cite[Lemma 4.13]{OinertRichterSilvestrov2013}. For a graded ideal $J$ of $A$ and $n\ge0$ let
\[
 L_n(J)=\Bigl\{a\in R:\ at^n+\sum_{i<n}a_it^i\in J\ \text{for some }a_i\in R\Bigr\}.
\]

\begin{lemma}\label{YAMA87KQWR}
$L_n(J)$ is a graded ideal of $R$.
\end{lemma}

\begin{proof}
$L_n(J)$ is closed under addition, since $J$ is. Let $a\in L_n(J)$, witnessed by $f=at^n+\sum_{i<n}a_it^i\in J$, and let $r\in R$. Then $rf\in J$ has $t^n$-coefficient $ra$, so $ra\in L_n(J)$. By \eqref{YAMO47UXRT}, $t^i\sigma^{-n}(r)\in\sigma^{i-n}(r)t^i+\sum_{m<i}Rt^m$ for every $i\ge0$. Hence $f\sigma^{-n}(r)\in J$ has $t$-degree at most $n$ and $t^n$-coefficient $ar$, so $ar\in L_n(J)$.

For gradedness, write $a=\sum_\tau a_\tau$. By \eqref{YAMF65QNWT}, the component of $f$ of degree $\tau\gamma^n$ lies in $J$, has $t$-degree at most $n$, and has $t^n$-coefficient $a_\tau$. So $a_\tau\in L_n(J)$ for every $\tau$.
\end{proof}

\begin{lemma}\label{YAMX41TVNZ}
For $n\ge1$ and $r\in R$,
\[
 t^nr=\sigma^n(r)t^n+D_n(r)t^{n-1}+\sum_{m<n-1}Rt^m,\quad D_n=\sum_{j=0}^{n-1}\sigma^j\delta\sigma^{n-1-j}.
\]
\end{lemma}

\begin{proof}
For $n=1$ this is \eqref{YAMO47UXRT}. Assume the formula for $n$ and multiply it on the left by $t$. By \eqref{YAMO47UXRT},
\begin{align*}
 t^{n+1}r=\sigma^{n+1}(r)t^{n+1}+\bigl(\delta\sigma^n(r)+\sigma D_n(r)\bigr)t^n+\sum_{m<n}Rt^m,
\end{align*}
since
\begin{align*}
 t\,\bigl(D_n(r)t^{n-1}\bigr)=\sigma D_n(r)t^n+\delta D_n(r)t^{n-1}
\end{align*}
and the remaining terms have $t$-degree at most $n-1$. Finally,
\begin{align*}
 \delta\sigma^n+\sigma D_n=\sum_{j=0}^{n}\sigma^j\delta\sigma^{n-j}=D_{n+1}.
\end{align*}
\end{proof}

The next lemma records basic properties of monic normal elements.

\begin{lemma}\label{YAMG15WRKP}
Let $n\ge1$ and let $f=t^n+at^{n-1}+\sum_{i<n-1}a_it^i$ be homogeneous. Then $f\in A_{\gamma^n}$ and $a\in R_\gamma$. Moreover:
\begin{enumerate}[label=\textup{(\arabic*)}]
\item $f$ is normal if and only if $fr=\sigma^n(r)f$ for all $r\in R$ and $tf-ft\in Rf$. In that case $tf=ft+(\sigma(a)-a)f$.
\item If $f$ is normal, then $\Delta_n=\delta_{-a}$. In particular $\Delta_n$ is inner.
\item If $f$ is normal, then $Af$ is a nonzero proper graded ideal of $A$.
\end{enumerate}
\end{lemma}

\begin{proof}
The $t^n$-coefficient $1$ lies in $R_e$, so $f\in A_{\gamma^n}$ by \eqref{YAMF65QNWT}, and then $a\in R_{\gamma^n\gamma^{1-n}}=R_\gamma$. Since $f$ is monic, $\lc(gf)=\lc(g)$ for every nonzero $g\in A$, hence
\begin{equation}\label{YAME69MXQL}
 \degt(gf)=\degt(g)+n\quad (g\neq0).
\end{equation}

(1) Suppose $f$ is normal and let $r\in R$. Write $fr=gf$ with $g\in A$. By \eqref{YAME69MXQL}, $\degt g\le0$, so $g\in R$, and comparing $t^n$-coefficients via \cref{YAMX41TVNZ} gives $g=\sigma^n(r)$. Similarly $ft=hf$ with $\degt h=1$, and comparing $t^{n+1}$-coefficients gives $h=t+h_0$ with $h_0\in R$, so $tf-ft=-h_0f\in Rf$. Conversely, assume both conditions and write $tf-ft=cf$ with $c\in R$. Then $fr=\sigma^n(r)f$ and $ft=(t-c)f$, so $fA\subseteq Af$ because $R$ and $t$ generate $A$. Since $\sigma^n$ is bijective, $rf=f\sigma^{-n}(r)$ and $tf=ft+f\sigma^{-n}(c)$, so $Af\subseteq fA$. Finally $tf=t^{n+1}+\sigma(a)t^n+\cdots$ and $ft=t^{n+1}+at^n+\cdots$, and comparing $t^n$-coefficients in $tf-ft=cf$ gives $c=\sigma(a)-a$.

(2) By (1), $fr=\sigma^n(r)f$. Compare $t^{n-1}$-coefficients. By \cref{YAMX41TVNZ}, $t^nr$ contributes $D_n(r)$ and $at^{n-1}r$ contributes $a\sigma^{n-1}(r)$, while the terms $a_it^ir$ with $i<n-1$ do not contribute. On the right, $\sigma^n(r)f$ contributes $\sigma^n(r)a$. Hence
\[
 D_n(r)+a\sigma^{n-1}(r)=\sigma^n(r)a.
\]
Substitute $r=\sigma^{1-n}(y)$. Since $D_n\sigma^{1-n}=\sum_{j=0}^{n-1}\sigma^j\delta\sigma^{-j}=\Delta_n$, this becomes
\[
 \Delta_n(y)=\sigma(y)a-ay=(-a)y-\sigma(y)(-a)=\delta_{-a}(y).
\]

(3) $Af=fA$ is an ideal, graded because $f$ is homogeneous, and nonzero. By \eqref{YAME69MXQL} its nonzero elements have $t$-degree at least $n\ge1$, so $1\notin Af$.
\end{proof}

\begin{proof}[\textnormal{\textbf{Proof of \cref{YAMJ76MXRP}(1)}}]
If $A$ contains a monic normal element, $A$ is not gr-simple by \cref{YAMG15WRKP}(3). Conversely, let $J\neq0$ be a graded ideal of $A$, and let $n$ be the least $t$-degree of a nonzero element of $J$. By \cref{YAMA87KQWR}, $L_n(J)$ is a nonzero graded ideal of $R$, so $L_n(J)=R$. Choose $g\in J$ of $t$-degree $n$ with $t^n$-coefficient $1$. If $n=0$, then $g=1\in J$ and $J=A$. Let $n\ge1$. Since $1\in R_e$, the component of $g$ of degree $\gamma^n$ lies in $J$ and has the form
\[
 f=t^n+at^{n-1}+\sum_{i<n-1}a_it^i.
\]
For $r\in R$, both $fr$ and $\sigma^n(r)f$ have $t^n$-coefficient $\sigma^n(r)$. So $fr-\sigma^n(r)f\in J$ has $t$-degree less than $n$, and it vanishes by minimality of $n$. Likewise, in $tf-ft-(\sigma(a)-a)f\in J$ the $t^{n+1}$-terms cancel and the $t^n$-coefficients agree, so this element vanishes. By \cref{YAMG15WRKP}(1), $f$ is a monic normal element. Hence, if $A$ contains no monic normal element, then $J=A$.
\end{proof}

The same argument also determines the generators of graded ideals.

\begin{corollary}\label{YAMJ3YJS}
Assume that $R$ is gr-simple and $\sigma$ is an automorphism. Every nonzero graded ideal $J$ of $A$ has the form $J=Af=fA$ for a unique homogeneous normal polynomial $f$ that is monic. For $J=A$ this polynomial is $1$, and for a proper ideal it has positive $t$-degree.
\end{corollary}

\begin{proof}
Let $n$ be the least $t$-degree of a nonzero element of $J$. If $n=0$, the proof of \cref{YAMJ76MXRP}(1) gives $J=A$, and we take $f=1$. If $n\geq1$, that proof constructs a homogeneous monic normal polynomial $f\in J$ of $t$-degree $n$.

Let $0\neq g\in J$, with $\degt g=m\geq n$ and $\lc(g)=a$. Since $f$ is monic, $g-at^{m-n}f$ belongs to $J$ and has $t$-degree less than $m$, unless it is zero. Repeating this step gives $g=qf+r$, where $r=0$ or $\degt r<n$. As $r\in J$, minimality of $n$ gives $r=0$. Hence $J=Af=fA$.

If $f'$ is another monic generator, write $f'=af$ and $f=bf'$ with $a,b\in A$. Multiplication on the right by a monic polynomial adds its $t$-degree, so these two equalities imply $\degt f'=\degt f$ and $a\in R$. Comparing leading coefficients gives $a=1$, and thus $f'=f$.
\end{proof}

For the trivial grading, Lam and Leroy give necessary and sufficient conditions for an Ore extension to be simple \cite[Theorems 4.5 and 5.8]{LamLeroy1992} (see also \cite{OinertRichterSilvestrov2013}). Parts (2) and (3) of \cref{YAMJ76MXRP} give conditions on $\delta$ that exclude monic normal elements. For the trivial grading, they are simplicity criteria over simple coefficient rings.

\begin{proof}[\textnormal{\textbf{Proof of \cref{YAMJ76MXRP}(2)}}]
Combine \cref{YAMJ76MXRP}(1) and \cref{YAMG15WRKP}(2).
\end{proof}

The argument of \cref{YAMG15WRKP}(2) is the one in \cite[proof of Lemma 4.8]{eisenschlos2013}. There the coefficient ring is a locally finite connected graded algebra generated in degree one, without zero divisors in degree one, whose only normal elements are the scalars, $\sigma$ is a graded automorphism and $\delta$ is homogeneous of degree one. If no $\Delta_n$ is inner, then the only normal elements of the Ore extension are the scalars.

\cref{YAMJ76MXRP}(2) never applies to $\sigma=\id$ if $p1_R=0$ for a prime $p$, since then $\Delta_p=p\delta=0$. \cref{YAMJ76MXRP}(1) has no such restriction, and for $\sigma=\id$ its argument works under a weaker hypothesis on $R$.

\begin{corollary}\label{YAMI54TXWR}
Let $\sigma=\id$, let $(\id,\delta)$ be $\gamma$-compatible, and assume that $R$ is $(\id,\delta)$-gr-simple. Then $R[t;\delta]$ is gr-simple if and only if it contains no homogeneous central element that is monic of positive $t$-degree.
\end{corollary}

\begin{proof}
A central element is normal, so one direction is \cref{YAMG15WRKP}(3). Conversely, let $J\neq0$ be a graded ideal and let $n$ be the least $t$-degree of a nonzero element of $J$. By \cref{YAMA87KQWR}, $L_n(J)$ is a nonzero graded ideal of $R$. If $f=\sum_ia_it^i\in J$ with $\degt f\le n$, then $tf-ft=\sum_i\delta(a_i)t^i\in J$, so $\delta(L_n(J))\subseteq L_n(J)$. Hence $L_n(J)=R$. If $n=0$, then $1\in J$. Let $n\ge1$. As in the proof of \cref{YAMJ76MXRP}(1), $J$ contains a homogeneous element $f=t^n+\sum_{i<n}a_it^i$. For $r\in R$, the element $fr-rf\in J$ has $t$-degree less than $n$, and so does $tf-ft=\sum_{i<n}\delta(a_i)t^i$, because $\delta(1)=0$. Both vanish by minimality of $n$, so $f$ is central.
\end{proof}

\cref{YAMI54TXWR} is the graded form of \cite[Lemma 4.13 and Proposition 4.14]{OinertRichterSilvestrov2013}. It needs no hypothesis on the centre of $R[t;\delta]$, in contrast with the centre criterion \cite[Theorem 1.10(3)]{AitMohamed2026GradedDifferential}, which assumes that this centre is a graded subring (this is automatic when $\Gamma$ is abelian). Compare also \cite[Theorem 4.15]{OinertRichterSilvestrov2013}: $R[t;\delta]$ is simple if and only if $R$ is $\delta$-simple and the centre of $R[t;\delta]$ is a field.

\begin{proof}[\textnormal{\textbf{Proof of \cref{YAMJ76MXRP}(3)}}]
Necessity is \cref{YAMT64XWRP}. For sufficiency, write $\sigma\delta\sigma^{-1}-\delta=\delta_v$. The pair $(\sigma,\sigma\delta\sigma^{-1}-\delta)$ is $\gamma$-compatible by \cref{YAMW28KQRP}(3), so we may take $v\in R_\gamma$ by \cref{YAMW28KQRP}(2).

We show by induction that $\sigma^j\delta\sigma^{-j}-\delta=\delta_{v_j}$ with $v_j\in R_\gamma$ for all $j\ge0$. For $j=0$ take $v_0=0$. If it holds for $j$, then, since $u\mapsto\delta_u$ is additive and $\sigma\delta_{v_j}\sigma^{-1}=\delta_{\sigma(v_j)}$ with $\sigma(v_j)\in R_\gamma$ by \cref{YAMW28KQRP}(3),
\begin{align*}
 \sigma^{j+1}\delta\sigma^{-j-1}-\delta
 &=\sigma\bigl(\sigma^j\delta\sigma^{-j}-\delta\bigr)\sigma^{-1}+\bigl(\sigma\delta\sigma^{-1}-\delta\bigr)\\
 &=\delta_{\sigma(v_j)+v}.
\end{align*}
Summing over $j=0,\ldots,n-1$ gives $\Delta_n=n\delta+\delta_{w_n}$ with $w_n=\sum_{j<n}v_j\in R_\gamma$.

Let $n\ge1$. The element $n1_R$ is central, homogeneous of degree $e$, and nonzero since $\operatorname{char}R=0$. So $(n1_R)R$ is a nonzero graded ideal, hence equal to $R$, and $n1_R$ is invertible with central inverse $n^{-1}$. If $\Delta_n=\delta_y$ for some $y\in R$, then $n\delta=\delta_{y-w_n}$, and multiplying by $n^{-1}$ gives $\delta=\delta_{n^{-1}(y-w_n)}$. This contradicts the outerness of $\delta$. Hence every $\Delta_n$ is outer, and \cref{YAMJ76MXRP}(2) applies.
\end{proof}

\begin{example}[Characteristic $p$]\label{YAMDSO0Q}
Let $\operatorname{char}k=p>0$ and let $R=k[x^{\pm1}]$ be $\Z$-graded with $\deg x=1$. Every nonzero homogeneous element of $R$ is a unit, so $R$ is gr-simple. Let $\delta=\frac{d}{dx}$. Then $\delta(R_n)\subseteq R_{n-1}$, so $A=R[t;\delta]$ is graded with $\deg t=-1$ by \cref{YAMK91QZWM}. Since $\binom pi=0$ in $k$ for $0<i<p$, we have
\[
t^pr=\sum_{i=0}^p\tbinom pi\delta^i(r)t^{p-i}=rt^p+\delta^p(r)\quad(r\in R).
\]
Moreover $\delta^p(x^n)=n(n-1)\cdots(n-p+1)x^{n-p}=0$ for all $n\in\Z$, because the product of $p$ consecutive integers is divisible by $p$. Hence $t^p$ is a homogeneous central element that is monic of $t$-degree $p$, and $A$ is not gr-simple by \cref{YAMG15WRKP}(3). In characteristic zero the same formulas give a gr-simple ring: $R$ is commutative, so every inner derivation of $R$ vanishes, $\delta\neq0$ is outer, and \cref{YAMJ76MXRP}(3) applies.
\end{example}

\begin{remark}\label{YAMR58CTLM}
Let $R$ be prime, $\sigma$ an automorphism and $\delta$ an $X$-outer $\sigma$-derivation, that is, $\delta$ is not $\delta_b$ with $b$ in the symmetric Martindale ring of quotients of $R$. Chuang and Tsai \cite{ChuangTsai2009} study semi-invariant polynomials, with coefficients in the symmetric Martindale ring of quotients. They show that, in characteristic zero, the minimal positive degree of a monic semi-invariant polynomial, when one exists, is the least integer $\nu\ge1$ for which
\[
\sigma^\nu\delta\sigma^{-\nu}-\delta
\]
is $X$-inner. The hypothesis of \cref{YAMJ76MXRP}(3) corresponds to $\nu=1$. There is no monic semi-invariant polynomial $t+b$ of degree one: if $(t+b)r=r'(t+b)$ for each $r$, then comparing $t$-coefficients gives $r'=\sigma(r)$, and comparing constant terms gives $\delta(r)=\sigma(r)b-br$, so $\delta$ would be $\delta_{-b}$, hence $X$-inner. For the trivial grading and $R$ simple, \cref{YAMJ76MXRP}(3) is therefore close to their result. It does not require $R$ prime: $kC_2$ is gr-simple but not prime (see \cref{YAMQQM8W}).
\end{remark}

\begin{remark}\label{YAMR27NECS}
We do not know whether the innerness hypothesis in \cref{YAMJ76MXRP}(3) can be removed. The proof needs it for the following reason. Suppose that $\sigma^2=\id$, $\sigma\delta\sigma^{-1}=-\delta$, and $\delta^2(r)=rb-br$ for all $r\in R$, where $b\in R_{\gamma^2}$ satisfies $\sigma(b)=b$ and $\delta(b)=0$. Then $D_2=\delta\sigma+\sigma\delta=0$ in \cref{YAMX41TVNZ}, so $t^2r=rt^2+\delta^2(r)$. The homogeneous element $f=t^2+b$ satisfies $fr=rf$ and $tf-ft=(\sigma(b)-b)t+\delta(b)=0$, so $A$ is not gr-simple by \cref{YAMG15WRKP}(3). On the other hand, if $R$ is gr-simple of characteristic zero, then $2\cdot1_R$ is a central unit (see the proof of \cref{YAMJ76MXRP}(3)), so $\sigma\delta\sigma^{-1}-\delta=-2\delta$ is outer whenever $\delta$ is. We do not know of such data with $R$ gr-simple of characteristic zero and $\delta$ outer. By \cref{YAMN47ZXVL}(3) and \cite[Theorem 1.10(2)]{AitMohamed2026GradedDifferential}, $\sigma$ would have to differ from conjugation by a unit of $R_\gamma$.
\end{remark}

\subsection{Group algebras}

In this subsection we use the notation of \cref{YAMZNTWS}. Taking $\tau=\rho=e$ in the cocycle identity gives $c(e)=0$, and then $c(\tau^{-1})=-\chi(\tau)^{-1}c(\tau)$.

\begin{proposition}\label{YAMM18QZKN}
Let $G$ be a group, $R=kG$ with $R_\tau=k\tau$, $\gamma\in G$, and $\chi\colon G\to k^\times$ a homomorphism. Put $\sigma(\tau)=\chi(\tau)\gamma\tau\gamma^{-1}$.
\begin{enumerate}[label=\textup{(\arabic*)}]
\item $\sigma$ extends to a $k$-algebra automorphism of $kG$ with $\sigma(R_\tau)=R_{\gamma\tau\gamma^{-1}}$.
\item The $k$-linear $\sigma$-derivations $\delta$ for which $(\sigma,\delta)$ is $\gamma$-compatible are exactly the maps $\delta_c(\tau)=c(\tau)\gamma\tau$ with $c\in Z^1(G,k_\chi)$.
\item $\delta_c$ is inner if and only if $c\in B^1(G,k_\chi)$.
\item $\sigma\delta_c\sigma^{-1}-\delta_c=-\delta_u$ with $u=c(\gamma)\gamma\in R_\gamma$. In particular it is inner.
\end{enumerate}
\end{proposition}

\begin{proof}
(1) For $\tau,\rho\in G$,
\begin{align*}
 \sigma(\tau)\sigma(\rho)=\chi(\tau)\chi(\rho)\gamma\tau\gamma^{-1}\gamma\rho\gamma^{-1}=\sigma(\tau\rho).
\end{align*}
So $\tau\mapsto\sigma(\tau)$ is a group homomorphism into the units of $kG$, and it extends $k$-linearly to an algebra endomorphism of $kG$. Since $k^\times$ is abelian, $\chi(\gamma^{-1}\tau\gamma)=\chi(\tau)$, and $\psi(\tau)=\chi(\tau)^{-1}\gamma^{-1}\tau\gamma$ satisfies $\sigma\psi(\tau)=\tau=\psi\sigma(\tau)$. Hence $\psi=\sigma^{-1}$. Finally, $\sigma(k\tau)=k\gamma\tau\gamma^{-1}$ because $\chi(\tau)\neq0$.

(2) Compatibility forces $\delta(\tau)\in R_{\gamma\tau}=k\gamma\tau$, say $\delta(\tau)=c(\tau)\gamma\tau$. Then
\begin{align*}
 \sigma(\tau)\delta(\rho)+\delta(\tau)\rho&=\chi(\tau)\gamma\tau\gamma^{-1}c(\rho)\gamma\rho+c(\tau)\gamma\tau\rho\\
 &=\bigl(c(\tau)+\chi(\tau)c(\rho)\bigr)\gamma\tau\rho.
\end{align*}
Since $G$ is a $k$-basis of $kG$, the identity $\delta(\tau\rho)=\sigma(\tau)\delta(\rho)+\delta(\tau)\rho$ is exactly the cocycle identity. Conversely, for a cocycle $c$ the $k$-linear map $\delta_c$ satisfies the Leibniz rule on the basis $G$, hence everywhere by bilinearity.

(3) By \cref{YAMW28KQRP}(2), $\delta_c$ is inner if and only if $\delta_c=\delta_u$ for some $u\in R_\gamma$. Write $u=x\gamma$ with $x\in k$. Then
\begin{align*}
 \delta_u(\tau)=x\gamma\tau-\chi(\tau)\gamma\tau\gamma^{-1}x\gamma=x(1-\chi(\tau))\gamma\tau.
\end{align*}
So $\delta_c$ is inner if and only if $c=-x(\chi-1)$ for some $x\in k$, that is, $c\in B^1(G,k_\chi)$.

(4) We have $\sigma^{-1}(\tau)=\chi(\tau)^{-1}\gamma^{-1}\tau\gamma$ and $\sigma(\tau\gamma)=\chi(\tau)\chi(\gamma)\gamma\tau$. Hence
\begin{align*}
 \sigma\delta_c\sigma^{-1}(\tau)&=\chi(\tau)^{-1}c(\gamma^{-1}\tau\gamma)\,\sigma(\tau\gamma)\\
 &=\chi(\gamma)\,c(\gamma^{-1}\tau\gamma)\,\gamma\tau.
\end{align*}
The cocycle identity, applied to $\gamma^{-1}\cdot(\tau\gamma)$ and to $\tau\cdot\gamma$, together with $c(\gamma^{-1})=-\chi(\gamma)^{-1}c(\gamma)$, gives
\begin{align*}
 \chi(\gamma)c(\gamma^{-1}\tau\gamma)&=-c(\gamma)+c(\tau)+\chi(\tau)c(\gamma)\\
 &=c(\tau)-(1-\chi(\tau))c(\gamma).
\end{align*}
By the formula in (3) with $x=c(\gamma)$, this reads $\sigma\delta_c\sigma^{-1}(\tau)=\delta_c(\tau)-\delta_u(\tau)$ with $u=c(\gamma)\gamma$.
\end{proof}

Every $\iota_\gamma$-graded $k$-algebra endomorphism $\sigma$ of $kG$ has the form $\tau\mapsto\chi(\tau)\gamma\tau\gamma^{-1}$ for a character $\chi$: the element $\sigma(\tau)\in k\gamma\tau\gamma^{-1}$ is nonzero because $\sigma(\tau)\sigma(\tau^{-1})=1$, and multiplicativity makes $\chi$ a homomorphism. Hence \cref{YAMM18QZKN} describes all $\gamma$-compatible pairs of $k$-linear maps on $kG$. If $\chi=1$, then $\sigma$ is conjugation by the homogeneous unit $\gamma$, and \cref{YAMN47ZXVL}(3) identifies $kG[t;\sigma,\delta_c]$ with the differential polynomial ring $kG[t';\delta']$, where $\deg t'=e$, $\delta'(\tau)=c(\tau)\tau$ and $c\in\operatorname{Hom}(G,k)$. If $\chi\neq1$, then $\sigma$ is not conjugation by a homogeneous unit: the homogeneous units are the elements $\lambda\rho$ with $\lambda\in k^\times$, $\rho\in G$, and conjugation by $\lambda\rho$ maps $\tau$ to $\rho\tau\rho^{-1}$, whereas $\sigma(\tau)=\chi(\tau)\gamma\tau\gamma^{-1}$.

The ring $kG$ is gr-simple, because every nonzero homogeneous element is a unit. By \cref{YAMM18QZKN}(4), \cref{YAMJ76MXRP}(3) therefore applies to the family of \cref{YAMM18QZKN}.

\begin{corollary}\label{YAMO31XMKT}
Let $\operatorname{char}k=0$, and let $\gamma$, $\chi$, $\sigma$ and $c\in Z^1(G,k_\chi)$ be as in \cref{YAMM18QZKN}. Then $kG[t;\sigma,\delta_c]$, graded by $\deg t=\gamma$, is gr-simple if and only if $[c]\neq0$ in $H^1(G,k_\chi)$.
\end{corollary}

\begin{proof}
Combine \cref{YAMJ76MXRP}(3) with \cref{YAMM18QZKN}(3),(4).
\end{proof}

The cohomological criterion can be tested on generators.

\begin{corollary}\label{YAMXVCBR}
In the setting of \cref{YAMO31XMKT}, suppose that $G$ is generated by $g_1,\ldots,g_m$. If $\chi=1$, then $kG[t;\sigma,\delta_c]$ is gr-simple if and only if $c(g_i)\neq0$ for some $i$. If $\chi\neq1$, it is gr-simple if and only if
\[
 c(g_i)\bigl(\chi(g_j)-1\bigr)-c(g_j)\bigl(\chi(g_i)-1\bigr)\neq0\quad\text{for some }i,j.
\]
\end{corollary}

\begin{proof}
By the cocycle identity and $c(g^{-1})=-\chi(g)^{-1}c(g)$, two cocycles that agree on the generators agree on $G$. So $[c]=0$ if and only if $c(g_i)=x(\chi(g_i)-1)$ for all $i$ and some $x\in k$. If $\chi=1$, this means $c(g_i)=0$ for all $i$. If $\chi\neq1$, some $\chi(g_i)-1$ is nonzero, and the vector $(c(g_i))_i$ is a multiple of $(\chi(g_i)-1)_i$ if and only if all the displayed $2\times2$ minors vanish. Now apply \cref{YAMO31XMKT}.
\end{proof}

For a presentation $G=\langle g_1,\ldots,g_m\mid w_1,\ldots,w_r\rangle$, the possible values $(c(g_i))_i$ are the solutions of a linear system. Let $F$ be the free group on $g_1,\ldots,g_m$, let $\pi\colon F\to G$ be the projection and $\tilde\chi=\chi\circ\pi$. Every vector $(u_i)\in k^m$ is $(\tilde c(g_i))_i$ for a unique cocycle $\tilde c\in Z^1(F,k_{\tilde\chi})$. This cocycle factors through $G$ if and only if $\tilde c(w_j)=0$ for every relator. Indeed, $\tilde\chi(w_j)=1$, so $\tilde c(gw_jg^{-1})=\tilde\chi(g)\tilde c(w_j)$ for every $g\in F$. Hence $\tilde c$ vanishes on the normal closure $N$ of the relators, and $\tilde c(gn)=\tilde c(g)$ for $n\in N$. \cref{YAMXVCBR} then tests gr-simplicity on the solution space.

\begin{remark}\label{YAMV47PQZN}
If $\chi(z)\neq1$ for some $z$ in the centre of $G$, then $H^1(G,k_\chi)=0$. Indeed, applying the cocycle identity to $z\tau=\tau z$ gives $c(z)+\chi(z)c(\tau)=c(\tau)+\chi(\tau)c(z)$, that is, $c(\tau)=x(\chi(\tau)-1)$ with $x=c(z)(\chi(z)-1)^{-1}$. So every cocycle is a coboundary. In particular, if $G$ is abelian and $\chi\neq1$, then $\delta_c$ is inner and no ring in the family is gr-simple (\cref{YAMT64XWRP}). Hence, if $\chi\neq1$ and $kG[t;\sigma,\delta_c]$ is gr-simple, then $G$ is nonabelian and $\chi$ is trivial on the centre of $G$.
\end{remark}

\begin{example}\label{YAMK82RTXW}
Let $G=F(a,b)$ be the free group on $a,b$, let $k=\Q$, fix $q\in\Q^\times$ with $q\neq1$, and take $\gamma=a$, $\chi(a)=q$, $\chi(b)=1$. A cocycle on a free group is determined by arbitrary values on the free generators. Let $c$ be the cocycle with $c(a)=0$ and $c(b)=1$. From $tr=\sigma(r)t+\delta_c(r)$ we get
\[
 ta=\chi(a)\,aaa^{-1}t+c(a)\,aa=q\,at,\quad
 tb=\chi(b)\,aba^{-1}t+c(b)\,ab=aba^{-1}t+ab.
\]
So the ring is generated by $\Q F(a,b)$ and $t$ subject to these relations, with $\deg t=a$. Every coboundary $x(\chi-1)$ vanishes at $b$, so $[c]\neq0$, and the ring is gr-simple by \cref{YAMO31XMKT}. More precisely, cocycles correspond to pairs $(c(a),c(b))\in\Q^2$ and coboundaries to the pairs $(x(q-1),0)$. Hence $H^1(F(a,b),\Q_\chi)\cong\Q$ via $[c]\mapsto c(b)$, and the ring attached to $c$ is gr-simple exactly when $c(b)\neq0$. Since $c(\gamma)=c(a)=0$, \cref{YAMM18QZKN}(4) gives $\sigma\delta_c=\delta_c\sigma$, the commuting case of \cref{YAMJ76MXRP}(3). Here $\sigma(R_b)=R_{aba^{-1}}\neq R_b$, and since $a$ does not commute with $b$, no differential polynomial ring over $\Q F(a,b)$ admits a compatible grading with $\deg t=a$ (\cref{YAMN47ZXVL}(2)). The degree can nevertheless be moved to $e$ by \cref{YAMN47ZXVL}(3) with $g=a$.
\end{example}

\subsection{Group algebras as skew group rings}\label{YAMG41AFFN}

For group algebras every $R_\gamma=k\gamma$ contains the unit $\gamma$. So, by \cref{YAMN47ZXVL}(3), the degree of the Ore variable can be moved to $e$, and the extension becomes a skew group ring over a polynomial ring. In particular, the noncentral degree of $t$ in \cref{YAMK82RTXW} comes from the choice of the variable $t$.

Keep the notation of \cref{YAMM18QZKN} and put $A=kG[t;\sigma,\delta_c]$. For $\tau\in G$ let
\[
 \varphi_\tau(X)=\chi(\tau)X+c(\tau)\in k[X].
\]
Since
\begin{align*}
 \varphi_\tau(\varphi_\rho(X))=\chi(\tau\rho)X+c(\tau)+\chi(\tau)c(\rho),
\end{align*}
the cocycle identity says exactly that $\varphi_{\tau\rho}=\varphi_\tau\circ\varphi_\rho$. Thus $\Phi\colon\tau\mapsto\varphi_\tau$ is a homomorphism from $G$ to $\mathrm{Aff}(1,k)$, and every homomorphism $G\to\mathrm{Aff}(1,k)$ arises in this way from a unique pair $(\chi,c)$ with $c\in Z^1(G,k_\chi)$. Moreover, $c=x(\chi-1)$ if and only if $\varphi_\tau(-x)=-x$ for all $\tau$. Hence $[c]=0$ if and only if $\Phi(G)$ has a fixed point in $k$. Let $G$ act on $k[X]$ by the $k$-algebra automorphisms $\theta_\tau(f)=f\circ\varphi_\tau^{-1}$, and let $k[X]\rtimes_\theta G=\bigoplus_{\tau}k[X]u_\tau$ be the corresponding skew group ring, so that $u_\tau f=\theta_\tau(f)u_\tau$ and $\deg u_\tau=\tau$.

\begin{proposition}\label{YAMP61AFMD}
Put $s=\gamma^{-1}t\in A_e$. Then $s\tau=\tau\varphi_\tau(s)$ for $\tau\in G$, $A_e=k[s]$ is a polynomial ring in $s$, and $A_\tau=\tau k[s]$ for every $\tau\in G$. The map $f(X)u_\tau\mapsto f(s)\tau$ is an isomorphism of $G$-graded $k$-algebras $k[X]\rtimes_\theta G\to A$.
\end{proposition}

\begin{proof}
Since $\chi(\tau)$ and $c(\tau)$ are central scalars,
\begin{align*}
 s\tau&=\gamma^{-1}\bigl(\chi(\tau)\gamma\tau\gamma^{-1}t+c(\tau)\gamma\tau\bigr)\\
 &=\tau\bigl(\chi(\tau)s+c(\tau)\bigr)=\tau\varphi_\tau(s).
\end{align*}
By \eqref{YAMO47UXRT}, $t\gamma^{-n}\in\chi(\gamma)^{-n}\gamma^{-n}t+R$, so by induction $s^n=\lambda_n\gamma^{-n}t^n+(\text{lower terms})$ with $\lambda_n\in k^\times$. Since $A_e=\bigoplus_nk\gamma^{-n}t^n$ by \eqref{YAMF65QNWT} and $s^n\in A_e$, the powers $s^n$ form a $k$-basis of $A_e$. Since $\tau$ is a unit in $A_\tau$, $\tau A_e\subseteq A_\tau$ and $\tau^{-1}A_\tau\subseteq A_e$, so $A_\tau=\tau k[s]$. From $\tau^{-1}s\tau=\varphi_\tau(s)$ we get $\tau f(s)\tau^{-1}=f(\varphi_\tau^{-1}(s))=\theta_\tau(f)(s)$. Hence
\[
 f(s)\tau\cdot g(s)\rho=f(s)\,\theta_\tau(g)(s)\,\tau\rho,
\]
which is the image of $(fu_\tau)(gu_\rho)=f\theta_\tau(g)u_{\tau\rho}$. So the map is multiplicative, and it maps $k[X]u_\tau$ bijectively onto $k[s]\tau=\tau k[s]=A_\tau$.
\end{proof}

Skew group rings of polynomial rings also appear in the study of iterated Hopf--Ore extensions of group algebras \cite{HatipogluLomp2026}.

\begin{proof}[\textnormal{\textbf{Proof of \cref{YAMD53GRPA}}}]
Part (1) is \cref{YAMP61AFMD}. Put $s=\gamma^{-1}t$.

(2) Let $J$ be a graded ideal and $I=J\cap A_e$, an ideal of $A_e=k[s]$. For $\tau\in G$, $\tau I\subseteq J_\tau$ and $\tau^{-1}J_\tau\subseteq I$, so $J=\bigoplus_\tau\tau I$. Moreover, $\tau I\tau^{-1}\subseteq J\cap A_e=I$. Conversely, if $I$ is an ideal of $k[s]$ with $\tau I\tau^{-1}\subseteq I$ for all $\tau$, then $\bigoplus_\tau\tau I$ is an ideal of $A$, because $\rho g(s)\tau I=\rho\tau(\tau^{-1}g(s)\tau)I\subseteq\rho\tau I$ and $\tau I\rho g(s)=\tau\rho(\rho^{-1}I\rho)g(s)\subseteq\tau\rho I$. Thus $J\mapsto J\cap A_e$ is a bijection onto the ideals $I$ of $k[s]$ with $\tau I\tau^{-1}\subseteq I$ for all $\tau$. Write $I=f(s)k[s]$ with $f=0$ or $f$ monic. As $\tau f(s)\tau^{-1}=f(\varphi_\tau^{-1}(s))$ and $f\circ\varphi_\tau^{-1}$ has the same $X$-degree as $f$, the condition reads $f\circ\varphi_\tau^{-1}\in k^\times f$ for all $\tau$, equivalently $f\circ\varphi_\tau\in k^\times f$ for all $\tau$. Finally $\bigoplus_\tau\tau I=Af(s)$, and $Af(s)\cap A_e=f(s)k[s]$ determines the monic polynomial $f$.

(3) By (2), $A$ is gr-simple if and only if every monic $f$ with $f\circ\varphi_\tau\in k^\times f$ for all $\tau$ equals $1$. Let $f$ be such a polynomial of positive degree, and let $Z\subseteq\bar k$ be its set of roots. If $x\in Z$, then $f\circ\varphi_\tau$ vanishes at $\varphi_\tau^{-1}(x)$, hence so does $f$. Thus $Z$ is a finite nonempty $\Phi(G)$-stable set, and its orbits are finite. Conversely, let $O\subseteq\bar k$ be a finite orbit, and let $M$ be the finite set of minimal polynomials over $k$ of the elements of $O$. For $x\in O$ and $\tau\in G$, $m_x\circ\varphi_\tau$ is irreducible of degree $\deg m_x$, because $\varphi_\tau$ is an affine bijection defined over $k$, and it vanishes at $\varphi_\tau^{-1}(x)\in O$. Hence $m_x\circ\varphi_\tau\in k^\times m_{\varphi_\tau^{-1}(x)}$. The map $m_x\mapsto m_{\varphi_\tau^{-1}(x)}$ is well defined on $M$ and surjective, since $\varphi_\tau^{-1}$ permutes $O$, so it is a bijection of $M$. Hence $f=\prod_{m\in M}m$ is monic of positive degree with $f\circ\varphi_\tau\in k^\times f$ for all $\tau$.

For the second equivalence, a finite orbit exists if and only if some subgroup $H$ of finite index fixes a point $x\in\bar k$: take for $H$ the stabilizer of a point of the orbit. Conversely, the orbit of $x$ has at most $[G:H]$ elements. If $H$ fixes $x$, then $c(h)=x(1-\chi(h))$ for $h\in H$. If $\chi|_H\neq1$, choosing $h_0\in H$ with $\chi(h_0)\neq1$ gives $x=c(h_0)/(1-\chi(h_0))\in k$, so $c|_H=(-x)(\chi-1)$. If $\chi|_H=1$, then $c|_H=0$. In both cases the restriction of $[c]$ to $H$ is zero. Conversely, if $c|_H=y(\chi-1)$ with $y\in k$, then $H$ fixes $-y$.

Now let $\operatorname{char}k=0$. Let $O\subseteq\bar k$ be a finite orbit, and put $b=|O|^{-1}\sum_{x\in O}x$, which is defined because $\operatorname{char}k=0$. Each $\varphi_\tau$ is affine and permutes $O$, so $\varphi_\tau(b)=b$. The argument above with $H=G$ gives $[c]=0$. Conversely, if $[c]=0$, then $\Phi(G)$ fixes a point of $k$, and its orbit is finite. Hence $A$ is gr-simple if and only if $[c]\neq0$.

(4) By (1), $A$ is the skew group ring $k[X]\rtimes_\theta G$ over the commutative ring $k[X]$. By \cite[Theorem 6.13]{Oinert2009}, $A$ is simple if and only if $k[X]$ is $G$-simple (that is, $0$ and $k[X]$ are its only $G$-stable ideals) and $k[X]$ is maximal commutative in $A$. By the proof of (2), $k[X]$ is $G$-simple if and only if $A$ is gr-simple. An element $\sum_\tau f_\tau u_\tau$ commutes with every $g\in k[X]$ if and only if $f_\tau(\theta_\tau(g)-g)=0$ for all $\tau$ and $g$. As $k[X]$ is a domain, $k[X]$ is maximal commutative if and only if $\theta_\tau\neq\id$ for $\tau\neq e$, that is, if and only if $\Phi$ is injective. Then $G\cong\Phi(G)\subseteq\mathrm{Aff}(1,k)=k\rtimes k^\times$, which is metabelian.
\end{proof}
The last assertion of \cref{YAMD53GRPA}(3) recovers \cref{YAMO31XMKT} without using \cref{YAMJ76MXRP}(3). The cohomological criterion fails in positive characteristic.

The following two affine actions illustrate the orbit criterion in \cref{YAMD53GRPA}(3). In panel (a), corresponding to \cref{YAME35FTHR}, solid arrows represent $X\mapsto X+1$ and dashed arrows represent $X\mapsto-X$. No point is fixed by the whole group, but $\mathbb F_3$ is a finite orbit, and $X^3-X$ gives the graded ideal $A(s^3-s)$. In panel (b), corresponding to \cref{YAME72BSLM}(2), solid arrows represent $X\mapsto X+1$ and dashed arrows represent $X\mapsto2X$. Only selected points and images are shown. Every orbit in $\bar{\Q}$ contains $x+\Z$ for any point $x$ in that orbit, and is therefore infinite. Since $G$ is defined as a subgroup of $\mathrm{Aff}(1,\Q)$, the action is faithful and the extension is simple.

\begin{figure}[ht]
\centering
\begin{tikzpicture}[>={Stealth[length=2mm]}, every node/.style={font=\small}]
\begin{scope}
\node[circle, draw, inner sep=1.5pt] (z) at (90:1.1) {$0$};
\node[circle, draw, inner sep=1.5pt] (o) at (210:1.1) {$1$};
\node[circle, draw, inner sep=1.5pt] (w) at (330:1.1) {$2$};
\draw[->, thick] (z) to[bend right=20] (o);
\draw[->, thick] (o) to[bend right=20] (w);
\draw[->, thick] (w) to[bend right=20] (z);
\draw[<->, dashed] (o) to[bend right=35] (w);
\draw[->, dashed] (z) to[out=60, in=120, looseness=6] (z);
\node at (0,-1.75) {(a) $\mathrm{Aff}(1,\mathbb F_3)$ on $\mathbb F_3$};
\end{scope}
\begin{scope}[xshift=3.2cm, yshift=-0.2cm]
\draw[->] (-0.4,0) -- (6.4,0);
\foreach \x/\l in {0/0, 0.75/\tfrac12, 1.5/1, 2.25/\tfrac32, 3/2, 4.5/3, 6/4} {
  \fill (\x,0) circle (1.6pt); \node at (\x,0.68) {$\l$};
}
\draw[->, thick,shorten >=2pt,shorten <=2pt] (0,0) to[bend left=55] (1.5,0);
\draw[->, thick,shorten >=2pt,shorten <=2pt] (1.5,0) to[bend left=55] (3,0);
\draw[->, thick,shorten >=2pt,shorten <=2pt] (3,0) to[bend left=55] (4.5,0);
\draw[->, thick,shorten >=2pt,shorten <=2pt] (4.5,0) to[bend left=55] (6,0);
\draw[->, dashed,shorten >=2pt,shorten <=2pt] (0.75,0) to[bend right=65] (1.5,0);
\draw[->, dashed,shorten >=2pt,shorten <=2pt] (1.5,0) to[bend right=55] (3,0);
\draw[->, dashed,shorten >=2pt,shorten <=2pt] (3,0) to[bend right=45] (6,0);
\node at (3,1.25) {$X\mapsto X+1$ (solid)};
\node at (3,-1.15) {$X\mapsto 2X$ (dashed)};
\node at (3,-1.85) {(b) Selected points in the $BS(1,2)$ action};
\end{scope}
\end{tikzpicture}
\end{figure}

\begin{example}\label{YAME35FTHR}
Let $k=\mathbb F_3$, $G=\mathrm{Aff}(1,\mathbb F_3)\cong S_3$, and let $\Phi$ be the inclusion, so that $\chi(\varphi)=a$ and $c(\varphi)=b$ for $\varphi(X)=aX+b$. No point of $\mathbb F_3$ is fixed by $G$, so $[c]\neq0$ in $H^1(G,k_\chi)$. However, $f=X^3-X$ satisfies $f(aX+b)=a^3X^3+b^3-aX-b=af(X)$ for $a,b\in\mathbb F_3$. By \cref{YAMD53GRPA}(2), $A(s^3-s)$ is a nonzero proper graded ideal of $A$ for every choice of $\gamma\in G$. So \cref{YAMO31XMKT} fails in characteristic $3$.
\end{example}

More generally, the cohomological criterion fails for every finite group with $H^1(G,k_\chi)\neq0$.

\begin{corollary}\label{YAMC19FING}
If $G$ is finite, then $kG[t;\sigma,\delta_c]$ is not gr-simple, for every $\gamma$, $\chi$ and $c$. For example, let $\operatorname{char}k=p>0$, $G=C_p=\langle g\rangle$, $\chi=1$, $\gamma\in G$ arbitrary and $c(g^i)=i\cdot1_k$. Then $[c]\neq0$ in $H^1(G,k)$, but $A(s^p-s)$, with $s=\gamma^{-1}t$ as in \cref{YAMP61AFMD}, is a nonzero proper graded ideal of $A$.
\end{corollary}

\begin{proof}
Every $\Phi(G)$-orbit in $\bar k$ has at most $|G|$ elements, so \cref{YAMD53GRPA}(3) applies. For $G=C_p$ and $\chi=1$, the map $c$ is a well-defined homomorphism because $p\cdot1_k=0$. Moreover, $B^1(G,k)=0$ and $c\neq0$, so $[c]\neq0$. Since $c(\tau)\in\mathbb F_p$, the polynomial $f=X^p-X$ satisfies $f\circ\varphi_\tau=f$ for all $\tau$, and \cref{YAMD53GRPA}(2) applies.
\end{proof}

In characteristic zero this agrees with \cref{YAMO31XMKT}, since $H^1(G,k_\chi)=0$ for finite $G$. In the example $\chi=1$, so $A$ is a differential polynomial ring by \cref{YAMN47ZXVL}(3). Thus the cohomological criterion already fails in the differential case.

\begin{remark}\label{YAMR46ORBT}
Let $s=\gamma^{-1}t$, $\lambda\in k$, $G_\lambda=\{\tau\in G:\varphi_\tau(\lambda)=\lambda\}$ and $L_\lambda=A(s-\lambda)$. Then $L_\lambda$ is a two-sided ideal if and only if $\lambda$ is fixed by $\Phi(G)$. Indeed, by \cref{YAMN47ZXVL}(3) with $g=\gamma$, $A=kG[s;\sigma',\delta']$ with $\sigma'(\tau)=\chi(\tau)\tau$ and $\delta'(\tau)=c(\tau)\tau$, so
\[
 (s-\lambda)\tau=\sigma'(\tau)(s-\lambda)+(\varphi_\tau(\lambda)-\lambda)\tau\quad(\tau\in G).
\]
Let $D$ be the $k$-linear map with $D(\tau)=(\varphi_\tau(\lambda)-\lambda)\tau$. Since $s-\lambda$ is monic, $A=kG\oplus L_\lambda$. For $a=r+g(s-\lambda)$ with $r\in kG$, we have $L_\lambda a\subseteq L_\lambda$ if and only if $(s-\lambda)r\in L_\lambda$, that is, $D(r)\in kG\cap L_\lambda=0$, that is, $r\in kG_\lambda$. So the idealizer $\{a\in A:L_\lambda a\subseteq L_\lambda\}$ equals $kG_\lambda\oplus L_\lambda$, and $L_\lambda$ is two-sided if and only if $G_\lambda=G$.

If $\operatorname{char}k=0$ and $A$ is not gr-simple, then $[c]=0$ by \cref{YAMD53GRPA}(3), so $\Phi(G)$ fixes some $\lambda\in k$ and $L_\lambda$ is a nonzero proper graded ideal. In \cref{YAME35FTHR} no such $\lambda$ exists: every $\lambda\in\mathbb F_3$ has stabilizer $\{X\mapsto aX+\lambda(1-a):a=\pm1\}$ of order $2$, so no $L_\lambda$ is two-sided, whereas $s^3-s=\prod_{\lambda\in\mathbb F_3}(s-\lambda)$ generates a proper graded ideal. The obstruction is carried by the orbit $\mathbb F_3$ rather than by a point.
\end{remark}

\begin{example}\label{YAME72BSLM}
(1) In \cref{YAMK82RTXW}, $\Phi(a)(X)=qX$ and $\Phi(b)(X)=X+1$. The orbit of any $x\in\bar k$ contains $x+\Z$, so the ring is gr-simple, as already shown. Since $F(a,b)$ is not solvable and $\mathrm{Aff}(1,\Q)$ is metabelian, $\Phi$ is not injective, and the ring is not simple by \cref{YAMD53GRPA}(4).

(2) Let $G$ be the subgroup of $\mathrm{Aff}(1,\Q)$ generated by $a(X)=2X$ and $b(X)=X+1$, let $\Phi$ be the inclusion, and let $\gamma=a$. Then $aba^{-1}=b^2$, and $G$ is isomorphic to the Baumslag--Solitar group $BS(1,2)=\langle a,b\mid aba^{-1}=b^2\rangle$. Here $\chi(a)=2$, $\chi(b)=1$, $c(a)=0$, $c(b)=1$, and the orbit of every $x\in\bar{\Q}$ contains $x+\Z$. Since $\sigma(b)=aba^{-1}=b^2$, the ring $\Q G[t;\sigma,\delta_c]$ is given by
\[
ta=2at,\quad tb=b^2t+ab,\quad \deg t=a,
\]
and it is simple by \cref{YAMD53GRPA}(3),(4).
\end{example}

\begin{example}[Translations]\label{YAMQJ76P}
Let $\operatorname{char}k=0$, $G=\Z^2$ with basis $g_1,g_2$, $\chi=1$, and $\gamma\in G$. Then $\sigma=\id$, since $G$ is abelian, and $Z^1(G,k_1)=H^1(G,k)=\operatorname{Hom}(G,k)$. For $c\in\operatorname{Hom}(G,k)$ put $c_i=c(g_i)$. The map $\Phi$ sends $\tau$ to the translation $X\mapsto X+c(\tau)$. By \cref{YAMD53GRPA}(3), $kG[t;\delta_c]$ is gr-simple if and only if $c\neq0$. By \cref{YAMD53GRPA}(4), it is simple if and only if $c$ is injective, that is, if and only if $c_1$ and $c_2$ are linearly independent over $\Q$. For $k=\mathbb R$ and $(c_1,c_2)=(1,\sqrt2)$ we obtain a simple ring. For $k=\Q$ and $(c_1,c_2)=(1,2)$ the ring is gr-simple but not simple. For $G=\Z$ and $c\neq0$ the ring is always simple.
\end{example}

\section{Poisson polynomial rings}\label{YAMP55QSCL}

A \emph{Poisson algebra} is a commutative $k$-algebra $P$ with a Lie bracket $\{-,-\}$ that is a derivation in each variable. It is \emph{$\Gamma$-graded} if $P_\rho P_\tau\subseteq P_{\rho\tau}$ and $\{P_\rho,P_\tau\}\subseteq P_{\rho\tau}$ for all $\rho,\tau$. A \emph{Poisson ideal} is an ideal $I$ with $\{P,I\}\subseteq I$, and $P$ is \emph{Poisson gr-simple} if $0$ and $P$ are its only graded Poisson ideals. An element $f\in P$ is \emph{Poisson normal} if $\{P,f\}\subseteq fP$. Then $Pf$ is a Poisson ideal.

A \emph{Poisson derivation} of $P$ is a derivation $\alpha$ with $\alpha(\{a,b\})=\{\alpha(a),b\}+\{a,\alpha(b)\}$. Given $\alpha$, a \emph{Poisson $\alpha$-derivation} is a derivation $\delta$ with
\[
 \delta(\{a,b\})=\{\delta(a),b\}+\{a,\delta(b)\}+\alpha(a)\delta(b)-\delta(a)\alpha(b)\quad(a,b\in P).
\]
For $u\in P$ put $\delta^{\mathrm p}_u(a)=\{u,a\}-u\alpha(a)$. We say that $\delta$ is \emph{$\alpha$-inner} if $\delta=\delta^{\mathrm p}_u$ for some $u\in P$. Finally, $P$ is \emph{$(\alpha,\delta)$-Poisson gr-simple} if $0$ and $P$ are its only graded Poisson ideals $I$ with $\alpha(I)+\delta(I)\subseteq I$.

In a commutative graded algebra, $P_\rho P_\tau\neq0$ forces $\rho\tau=\tau\rho$, and since
\begin{align*}
 \{P_\rho,P_\tau\}=\{P_\tau,P_\rho\}\subseteq P_{\rho\tau}\cap P_{\tau\rho},
\end{align*}
the same holds for nonzero brackets. In particular, if $P$ is a domain, then $\Supp(P)$ generates an abelian subgroup. If $P$ is not a domain, $\Supp(P)$ need not generate an abelian subgroup: $k[x,y]/(xy)$ is graded by the free group on $a,b$ with $\deg x=a$ and $\deg y=b$. However, if $P[t]$ is graded with $P_\tau$ in degree $\tau$ and $t$ in degree $\gamma$, then for $0\neq a\in P_\tau$ we have $0\neq ta\in P[t]_{\gamma\tau}\cap P[t]_{\tau\gamma}$, so $\gamma$ centralizes $\Supp(P)$ for any $\Gamma$. So the conjugation in \eqref{YAMB58LPXA} has no Poisson counterpart. Throughout this section, except in \cref{YAMR19NOGO}, $\Gamma$ is abelian.

Let $\alpha,\delta$ be derivations of $P$. By Oh's theorem \cite[Theorem 1.1]{Oh2006}, in the form and sign convention of \cite[Theorem 2.1]{LaunoisLecoutre2016}, there is a Poisson bracket on the polynomial algebra $P[t]$ that extends the bracket of $P$ and satisfies
\[
 \{t,a\}=\alpha(a)t+\delta(a)\quad(a\in P)
\]
if and only if $\alpha$ is a Poisson derivation and $\delta$ is a Poisson $\alpha$-derivation. Such a bracket is unique by the Leibniz rule. We write $B=P[t;\alpha,\delta]_{\mathrm p}$. A pair $(\alpha,\delta)$ is \emph{$\gamma$-compatible} if $\alpha(P_\tau)\subseteq P_\tau$ and $\delta(P_\tau)\subseteq P_{\gamma\tau}$ for all $\tau$. Since $\Gamma$ is abelian, these are the degree conditions in \cref{YAMPXB0E} with $\alpha$ in place of $\sigma$.

\begin{proof}[\textnormal{\textbf{Proof of \cref{YAME61POIS}(1)}}]
If such a grading exists and $a\in P_\tau$, then $\{t,a\}=\alpha(a)t+\delta(a)\in B_{\gamma\tau}$. Comparing the coefficients of $t$ and $1$ in each homogeneous component, as in the proof of \cref{YAMK91QZWM}, gives $\alpha(a)\in P_\tau$ and $\delta(a)\in P_{\gamma\tau}$. As in that proof, the grading is then
\begin{equation}\label{YAMP29GRDB}
 B_\beta=\bigoplus_{n\ge0}P_{\beta\gamma^{-n}}t^n\quad(\beta\in\Gamma).
\end{equation}
Conversely, let $(\alpha,\delta)$ be $\gamma$-compatible. Since $\Gamma$ is abelian, \eqref{YAMP29GRDB} defines a grading of the commutative algebra $P[t]$. Call a pair $(x,y)$ of homogeneous elements \emph{good} if $\{x,y\}\in B_{\deg x\deg y}$. Pairs of homogeneous elements of $P$ are good because $P$ is graded, the pairs $(t,a)$ and $(a,t)$ are good by compatibility and antisymmetry, and $(t,t)$ is good. If $(x,y)$ and $(x,z)$ are good, then so is $(x,yz)$, by the Leibniz rule and because $\Gamma$ is abelian. The same holds in the first entry. Every homogeneous element of $B$ is a sum of products of homogeneous elements of $P$ and copies of $t$, all of the same degree, so every pair of homogeneous elements is good.
\end{proof}

From now on $(\alpha,\delta)$ is $\gamma$-compatible and $B$ carries the grading \eqref{YAMP29GRDB}. As in the associative case, innerness produces a Poisson normal element of $t$-degree one.

\begin{lemma}\label{YAML47PINN}
Let $(\alpha,\delta)$ be $\gamma$-compatible.
\begin{enumerate}[label=\textup{(\arabic*)}]
\item If $\delta=\delta^{\mathrm p}_u$ with $u=\sum_\rho u_\rho$, then $\delta=\delta^{\mathrm p}_{u_\gamma}$.
\item If $\delta=\delta^{\mathrm p}_u$ with $u\in P_\gamma$, then $t-u\in B_\gamma$ is Poisson normal and $B(t-u)$ is a nonzero proper graded Poisson ideal of $B$.
\end{enumerate}
\end{lemma}

\begin{proof}
(1) For $a\in P_\tau$, the elements $\{u_\rho,a\}$ and $u_\rho\alpha(a)$ lie in $P_{\rho\tau}$, while $\delta(a)\in P_{\gamma\tau}$. Comparing components of degree $\gamma\tau$ gives the claim.

(2) For $a\in P$, $\{t-u,a\}=\alpha(a)t+\delta(a)-\{u,a\}=\alpha(a)(t-u)$. Moreover $\delta(u)=\{u,u\}-u\alpha(u)=-u\alpha(u)$, so $\{t,t-u\}=-\{t,u\}=-\alpha(u)t-\delta(u)=-\alpha(u)(t-u)$. Since $P$ and $t$ generate $B$, the Leibniz rule gives $\{B,t-u\}\subseteq(t-u)B$. Hence $B(t-u)$ is a Poisson ideal. It is graded because $t-u$ is homogeneous, and proper because its nonzero elements have positive $t$-degree.
\end{proof}

For $a\in P$ and $n\ge0$,
\begin{equation}\label{YAMQ38BRKT}
 \{a,t^n\}=-nt^{n-1}\bigl(\alpha(a)t+\delta(a)\bigr).
\end{equation}

\begin{proposition}\label{YAMT92PSMP}
If $B$ is Poisson gr-simple, then $P$ is $(\alpha,\delta)$-Poisson gr-simple and $\delta$ is not $\alpha$-inner.
\end{proposition}

\begin{proof}
Let $I$ be a graded Poisson ideal of $P$ with $\alpha(I)+\delta(I)\subseteq I$. Then $IB=\bigoplus_nIt^n$ is a graded ideal of $B$ with $IB\cap P=I$. For $a\in P$ and $i\in I$, \eqref{YAMQ38BRKT} gives $\{a,it^n\}=\{a,i\}t^n+i\{a,t^n\}\in IB$, and $\{t,it^n\}=(\alpha(i)t+\delta(i))t^n\in IB$. By the Leibniz rule $IB$ is a Poisson ideal, so $I\in\{0,P\}$. If $\delta$ is $\alpha$-inner, \cref{YAML47PINN} gives a nonzero proper graded Poisson ideal.
\end{proof}

\begin{proof}[\textnormal{\textbf{Proof of \cref{YAME61POIS}(2)}}]
If such an element $f$ exists, then $Bf$ is a nonzero proper graded Poisson ideal. Conversely, let $J\neq0$ be a graded Poisson ideal of $B$, let $n$ be the least $t$-degree of a nonzero element of $J$, and let $L\subseteq P$ be the set of $t^n$-coefficients of the elements of $J$ of $t$-degree at most $n$. As in \cref{YAMA87KQWR}, $L$ is a graded ideal of $P$. If $g=ct^n+\cdots\in J$ has $t$-degree at most $n$ and $a\in P$, then by \eqref{YAMQ38BRKT} the $t^n$-coefficient of $\{a,g\}\in J$ is $\{a,c\}-nc\alpha(a)$. Since $nc\alpha(a)\in L$, also $\{a,c\}\in L$. So $L$ is a nonzero graded Poisson ideal, and $L=P$. If $n=0$, then $1\in J$. Let $n\ge1$. The component of degree $\gamma^n$ of an element of $J$ with $t^n$-coefficient $1$ has the form $f=t^n+\sum_{i<n}b_it^i$ and lies in $J$. For $a\in P$, the element $\{a,f\}+n\alpha(a)f\in J$ has $t$-degree less than $n$, so it vanishes. The $t^n$-coefficient of $\{t,f\}=\sum_i(\alpha(b_i)t+\delta(b_i))t^i$ is $\alpha(b_{n-1})$, and its $t^{n+1}$-coefficient is $\alpha(1)=0$. So $\{t,f\}-\alpha(b_{n-1})f\in J$ vanishes as well. By the Leibniz rule $f$ is Poisson normal. Hence, if no such element exists, $J=B$.
\end{proof}

\begin{proof}[\textnormal{\textbf{Proof of \cref{YAME61POIS}(3)}}]
Necessity is \cref{YAMT92PSMP}. Conversely, suppose that $B$ is not Poisson gr-simple. By \cref{YAME61POIS}(2) there is a homogeneous Poisson normal $f=t^n+bt^{n-1}+\sum_{i<n-1}b_it^i$ with $n\ge1$. For $a\in P$ write $\{a,f\}=hf$. Comparing $t$-degrees gives $h\in P$, and the $t^n$-coefficients give $h=-n\alpha(a)$. By \eqref{YAMQ38BRKT}, the $t^{n-1}$-coefficient of $\{a,f\}$ is $-n\delta(a)+\{a,b\}-(n-1)b\alpha(a)$, while that of $-n\alpha(a)f$ is $-nb\alpha(a)$. Hence $n\delta(a)=\{a,b\}+b\alpha(a)$, that is, $\delta=\delta^{\mathrm p}_u$ with $u=-n^{-1}b$.
\end{proof}

For a commutative noetherian $\mathbb C$-algebra $P$ with zero bracket and $\alpha=0$, Jordan \cite{Jordan2014Poisson} shows that the prime spectrum of $P[t;\delta]$ is homeomorphic to the Poisson prime spectrum of $P[t;\delta]_{\mathrm p}$.

\subsection{Semiclassical limits}

Let $S$ be a $k$-algebra and $h\in S$ a central non-zero-divisor such that $S/hS$ is commutative. Then $S/hS$ is a Poisson algebra with bracket $\{\bar a,\bar b\}=\overline{h^{-1}(ab-ba)}$, called the \emph{semiclassical limit} of $S$. Here $h$ plays the role of Planck's constant \cite[Sections 1.7 and 2.1]{Goodearl2010Semiclassical}. For Ore extensions in which $\sigma-\id$ and $\delta$ take values in $hR$, the semiclassical limit is a Poisson polynomial ring \cite[Proposition 4.1]{LaunoisLecoutre2016}, \cite[Theorem 2.4]{ChoOh2016}.

\begin{lemma}\label{YAML26SPEC}
Let $w\in S$, and let $\vartheta$ be an automorphism of $S$ such that $sw=w\vartheta(s)$ and $\vartheta(s)-s\in hS$ for all $s\in S$. Then $\bar w$ is Poisson normal in $S/hS$.
\end{lemma}

\begin{proof}
For $s\in S$, $sw-ws=w(\vartheta(s)-s)$, so
\begin{align*}
 \{\bar s,\bar w\}=\bar w\,\overline{h^{-1}(\vartheta(s)-s)}\in\bar w\,(S/hS).
\end{align*}
\end{proof}

Let $S=R[z;\sigma,\delta]$ be a $k$-algebra, and let $h\in R$ be central in $S$ and a non-zero-divisor, with $R/hR$ commutative, $(\sigma-\id)(R)\subseteq hR$ and $\delta(R)\subseteq hR$. Put $P=R/hR$ and
\[
 \alpha_1(\bar r)=\overline{h^{-1}(\sigma(r)-r)},\quad\delta_1(\bar r)=\overline{h^{-1}\delta(r)}.
\]

\begin{lemma}\label{YAML73SCLV}
\begin{enumerate}[label=\textup{(\arabic*)}]
\item $\alpha_1$ and $\delta_1$ are well defined, $S/hS$ is the polynomial ring $P[\bar z]$, and $\{\bar z,\bar r\}=\alpha_1(\bar r)\bar z+\delta_1(\bar r)$. Hence $S/hS=P[\bar z;\alpha_1,\delta_1]_{\mathrm p}$.
\item For $u\in R$, we have $\delta_u(R)\subseteq hR$ and $\overline{h^{-1}\delta_u(r)}=\delta^{\mathrm p}_{\bar u}(\bar r)$, where $\delta^{\mathrm p}_{\bar u}$ is computed with $\alpha_1$. In particular, if $\delta=\delta_u$, then $\delta_1=\delta^{\mathrm p}_{\bar u}$ is $\alpha_1$-inner.
\end{enumerate}
\end{lemma}

\begin{proof}
(1) Since $h$ is central, $zh=hz$, that is, $\sigma(h)z+\delta(h)=hz$. Hence $\sigma(h)=h$ and $\delta(h)=0$. For $r'\in R$ this gives $\sigma(hr')-hr'=h(\sigma(r')-r')\in h^2R$ and $\delta(hr')=h\delta(r')\in h^2R$, so both formulas are independent of the representative, and division by $h$ is unique because $h$ is a non-zero-divisor. Since $h$ is central and lies in $R$, $hS=\bigoplus_nhRz^n$, so $S/hS=\bigoplus_nP\bar z^n$. This ring is commutative, because $\bar z\bar r-\bar r\bar z=\overline{(\sigma(r)-r)z+\delta(r)}=0$. Its semiclassical bracket extends that of $P$ and satisfies
\begin{align*}
 \{\bar z,\bar r\}&=\overline{h^{-1}(zr-rz)}\\
 &=\overline{h^{-1}(\sigma(r)-r)}\,\bar z+\overline{h^{-1}\delta(r)}\\
 &=\alpha_1(\bar r)\bar z+\delta_1(\bar r).
\end{align*}
By \cite[Theorem 2.1]{LaunoisLecoutre2016}, $\alpha_1$ is a Poisson derivation, $\delta_1$ is a Poisson $\alpha_1$-derivation, and $S/hS=P[\bar z;\alpha_1,\delta_1]_{\mathrm p}$. Compare \cite[Proposition 4.1]{LaunoisLecoutre2016} and \cite[Theorem 2.4]{ChoOh2016}.

(2) We have $\delta_u(r)=(ur-ru)-(\sigma(r)-r)u$, the first term lies in $hR$ because $R/hR$ is commutative, and the second because $(\sigma-\id)(R)\subseteq hR$. Hence
\[
 \overline{h^{-1}\delta_u(r)}=\{\bar u,\bar r\}-\alpha_1(\bar r)\bar u=\delta^{\mathrm p}_{\bar u}(\bar r).
\]
If $\delta=\delta_u$, then $\delta_1(\bar r)=\delta^{\mathrm p}_{\bar u}(\bar r)$ by this identity.
\end{proof}

\begin{example}\label{YAME58QWSL}
Let $F=\Q[q^{\pm1}]$, $h=q-1$ and $R=F[x]$, graded by $\deg x=1$ and $\deg q=0$. Let $\sigma$ be the $F$-algebra automorphism with $\sigma(x)=qx$, and let $\delta$ be the $F$-linear $\sigma$-derivation given by $\delta(f)=(f(qx)-f(x))/x$, so that $\delta(x^m)=(q^m-1)x^{m-1}$ for $m\ge1$, and put $S=R[z;\sigma,\delta]$, graded by $\deg z=-1$. Thus
\[
 zx=qxz+(q-1),
\]
and $h$ is central in $S$ because $\sigma(q)=q$ and $\delta(q)=0$. Specializing $q$ to $2$ gives the quantized Weyl algebra of \cref{YAMC56NPQW}. In $S$, the element $w=1+xz$ satisfies $wx=qxw$ and $zw=qwz$. Hence $sw=w\vartheta(s)$ for the $F$-algebra automorphism $\vartheta$ of $S$ with $\vartheta(x)=q^{-1}x$ and $\vartheta(z)=qz$, which preserves the defining relation. Since $\vartheta(x)-x$ and $\vartheta(z)-z$ lie in $hS$ and $\vartheta(ab)-ab=(\vartheta(a)-a)\vartheta(b)+a(\vartheta(b)-b)$, we get $\vartheta(s)-s\in hS$ for every $s\in S$.

The semiclassical limit $S/hS$ is $\Q[x,z]$ with $\{z,x\}=xz+1$, that is, $B=\Q[x][z;\alpha_1,\delta_1]_{\mathrm p}$ with $\alpha_1=x\frac{d}{dx}$ and $\delta_1=\frac{d}{dx}$, and $(\alpha_1,\delta_1)$ is $(-1)$-compatible. By \cref{YAML26SPEC}, $\bar w=1+xz\in B_0$ is Poisson normal. It is not a unit, so $B\bar w$ is a nonzero proper graded Poisson ideal and $B$ is not Poisson gr-simple. On the other hand, $\Q[x]$ with the zero bracket is $(\alpha_1,\delta_1)$-Poisson gr-simple, since no ideal $(x^m)$ with $m\ge1$ is $\delta_1$-stable, and $\delta_1$ is not $\alpha_1$-inner, because $\Q[x]_{-1}=0$ and $\delta_1\neq0$ (\cref{YAML47PINN}(1)). This is the Poisson counterpart of \cref{YAMQ95LNVK}. As in \cref{YAMZ32PXMT}, $w=x(z-u)$ with $u=-x^{-1}$. Over $F[x^{\pm1}]$ we have $\delta=\delta_u$, and by \cref{YAML73SCLV}(2), applied to $F[x^{\pm1}]$, $\delta_1=\delta^{\mathrm p}_{\bar u}$ is $\alpha_1$-inner over $\Q[x^{\pm1}]$. Thus the same factorization explains the normal element $1+xt$ of $A=S/(q-2)S$ and the Poisson normal element $1+xz$ of $B=S/(q-1)S$.
\end{example}

\begin{remark}\label{YAMR19NOGO}
The conjugation phenomenon of \cref{YAMK91QZWM} does not survive in a semiclassical limit. Let $S=R[z;\sigma,\delta]$ be graded as in \cref{YAMK91QZWM} with $\deg z=\gamma$, let $h\in R_e$ be central in $S$ and a non-zero-divisor, and suppose that $(\sigma-\id)(R)\subseteq hR$. If $\gamma\tau\gamma^{-1}\neq\tau$ and $r\in R_\tau$, then $\sigma(r)\in R_{\gamma\tau\gamma^{-1}}$, so $\sigma(r)-r\in hR$ has component $-r$ in degree $\tau$. Since $h$ is homogeneous, $hR$ is graded, and $r\in hR$. Hence $R_\tau\subseteq hR$, and the grading of $R/hR$ is supported on the centralizer of $\gamma$.
\end{remark}

\bibliographystyle{amsplain}
\bibliography{refs}

\end{document}